\documentclass[a4paper, 11pt]{article}

\usepackage{latexsym, amsfonts, amssymb, amsmath, graphicx, amsthm, enumerate, paralist, mathrsfs}
\usepackage{setspace, wrapfig, caption, subcaption, mathdots, comment, xspace}
\usepackage[T1]{fontenc}
\usepackage[english]{babel}
\usepackage[usenames,dvipsnames]{color}
\usepackage{hyperref}
\usepackage{pgf,tikz}
\usetikzlibrary{arrows}
\usepackage{tikz-cd, tkz-graph}

\newtheorem{coro}{Corollary}[section]
\newtheorem{thm}[coro]{Theorem}
\newtheorem{lem}[coro]{Lemma}
\newtheorem{prop}[coro]{Proposition}
\newtheorem*{mainthm}{Main Theorem}
\theoremstyle{definition}
\newtheorem{rem}[coro]{Remark}

\newtheorem{defn}[coro]{Definition}
\newtheorem{notn}[coro]{Notation}
\newtheorem{obs}[coro]{Observation}

\newcommand{\N}{\mathbb{N}}
\newcommand{\C}{\mathscr{C}}
\newcommand{\Z}{\mathbb{Z}}

\newcommand{\Rd}{{\mathcal{R}}}
\newcommand{\Pa}{{\mathcal{P}}}
\newcommand{\dw}{{\mathsf{d}_{W}}}
\newcommand{\ow}{\overline{w}}
\newcommand{\app}{\mathscr{A}_0}

\newcommand{\inv}[1]{#1^{-1}}
\newcommand{\adj}[1]{\overset{#1}{\sim}} 

\newcommand{\dash}{\nobreakdash-\hspace{0pt}}
\newcommand{\loccit}{\emph{loc.\@~cit.\@}\xspace}

\DeclareMathOperator{\Aut}{Aut}
\DeclareMathOperator{\ch}{Ch}
\DeclareMathOperator{\dist}{dist}
\DeclareMathOperator{\Fix}{Fix}
\DeclareMathOperator{\Flex}{Flex}
\DeclareMathOperator{\Pc}{Pc}
\DeclareMathOperator{\proj}{proj}
\DeclareMathOperator{\Rep}{Rep}

\DeclareMathOperator{\Sym}{Sym}
\DeclareMathOperator{\Stab}{Stab}
\DeclareMathOperator{\type}{type}

\DeclareMathOperator{\B}{\mathsf{B}}
\DeclareMathOperator{\Sp}{\mathsf{S}}

\title{Open subgroups of the automorphism group of a right-angled building}
\author{Tom De Medts \and Ana C. Silva}
\date{\today}

\begin{document}
\maketitle

\begin{abstract}
    We study the group of type-preserving automorphisms of a right-angled building, in particular when the building is locally finite.
    Our aim is to characterize the proper open subgroups as the
    finite index closed subgroups of the stabilizers of proper residues.

    One of the main tools is the new notion of firm elements in a right-angled Coxeter group, which are those elements for which the final letter
    in each reduced representation is the same.
    We also introduce the related notions of firmness for arbitrary elements of such a Coxeter group
    and $n$-flexibility of chambers in a right-angled building.
    These notions and their properties are used to determine the set of chambers fixed by the fixator of a ball.
    Our main result is obtained by combining these facts with ideas by Pierre-Emmanuel Caprace and Timoth\'ee Marquis
    in the context of Kac--Moody groups over finite fields,
    where we had to replace the notion of root groups by a new notion of root wing groups.
\end{abstract}

\section{Introduction}

A Coxeter group is right-angled if the entries of its Coxeter matrix are all equal to $1$, $2$ or $\infty$
(see Definition~\ref{def:cox} below for more details).
A right-angled building is a building for which the underlying Coxeter group is right-angled.
The most prominent examples of right-angled buildings are trees.
To some extent, the combinatorics of right-angled Coxeter groups and right-angled buildings
behave like the combinatorics of trees, but in a more complicated and therefore in many aspects more interesting fashion.

Right-angled buildings have received attention from very different perspectives.
One of the earlier motivations for their study was the connection with lattices,
starting with the important contributions of Marc Burger and Shahar Mozes on lattices in products of trees \cite{BM-lattices};
see, for instance, \cite{RR06, ThomasRAB, TW11, KT12, CT13}.
On the other hand, the automorphism groups of locally finite right-angled buildings are totally disconnected locally compact (t.d.l.c.) groups,
and their full automorphism group was shown to be an abstractly simple group by Pierre-Emmanuel Caprace in~\cite{Caprace},
making these groups valuable in the study of t.d.l.c.\@ groups.
Caprace's work also highlighted important combinatorial aspects of right-angled buildings; in particular, his study of parallel residues
and his notion of wings (see Definition~\ref{def:wings} below) are fundamental tools.
From this point of view, we have, in a joint work with Koen Struyve, introduced and investigated universal groups for right-angled buildings; see~\cite{DMSS}.
Another interesting aspect is their connection with spaces of non-positive curvature, beginning with the work of Marc Bourdon \cite{Bourdon}
and including the profound result of Mike Davis that buildings (with the appropriate metric realization) are always CAT(0) \cite{DavisCat0}.
More recently, Andreas Baudisch, Amador Martin-Pizarro and Martin Ziegler have studied right-angled buildings from a model-theoretic point of view;
see~\cite{model-RAB}.

\medskip

In this paper, we continue the study of right-angled buildings in a combinatorial and topological fashion.
In particular, we introduce some new tools in right-angled Coxeter groups and we study the (full) automorphism group of right-angled buildings.
Our main goal is to characterize the proper open subgroups of the automorphism group of a locally finite semi-regular right-angled building
as the closed finite index subgroups of the stabilizer of a proper residue. We prove our result in Theorem~\ref{th:openpropermain} below.

\begin{mainthm}
Let $\Delta$ be a thick irreducible semi-regular locally finite right-angled building of rank at least $2$.
Then any proper open subgroup of $\Aut(\Delta)$ is contained with finite index in the stabilizer in $\Aut(\Delta)$ of a proper residue.
\end{mainthm}

The first tool we introduce is the notion of \emph{firm elements} in a right-angled Coxeter group:
these are the elements with the property that every possible reduced representation of that element ends with the same letter
(see Definition~\ref{def:formF} below), i.e., the last letter cannot be moved away by elementary operations.
If an element of the Coxeter group is not firm, then we define its \emph{firmness} as the maximal length of a firm prefix.

This notion will be used to define the concepts of \emph{firm chambers} in a right-angled building
and of \emph{$n$-flexibility} of chambers with respect to another chamber; this then leads to the
notion of the \emph{$n$-flex} of a given chamber. See Definition~\ref{def:Ai} below.

A second new tool is the concept of a \emph{root wing group}, which we define in Definition~\ref{def:rootwinggroup}.
Strictly speaking, this is not a new definition since the root wing groups are defined as wing fixators, and as such they already appear in the work
of Caprace \cite{Caprace}.
However, we associate such a group to each \emph{root} (i.e., a half-apartment) of the building, and we explore the fact that they behave very much like root subgroups
in groups of a more algebraic nature, such as automorphism groups of Moufang spherical buildings or Kac--Moody groups.

\paragraph*{Outline of the paper.}

In Section~\ref{se:RAC}, we provide the necessary tools for right-angled Coxeter groups.
In Section~\ref{ss:poset}, we recall the notion of a poset $\prec_w$ that we can associate to any word $w$ in the generators, introduced in \cite{DMSS}.
Section~\ref{ss:firm} introduces the concepts of firm elements and the firmness of elements in a right-angled Coxeter group.
Our main result in that section is the fact that long elements cannot have a very low firmness; see Theorem~\ref{th:boundforN}.

Section~\ref{se:RAB} collects combinatorial facts about right-angled buildings.
After recalling the important notions of parallel residues and wings, due to Caprace \cite{Caprace}, in Section \ref{ss:RAB-pre},
we proceed in Section~\ref{ss:squares} to introduce the notion of chambers that are $n$-flexible with respect to another chamber
and the notion of the square closure of a set of chambers (which is based on results from~\cite{DMSS});
see Definitions~\ref{def:Ai} and~\ref{def:squareclosure}.
Our main result in Section~\ref{se:RAB} is Theorem~\ref{th:flex}, showing that the square closure of a ball of radius $n$ around a chamber $c_0$
is precisely the set of chambers that are $n$-flexible with respect to $c_0$.

In Section~\ref{se:AutRAB}, we study the automorphism group of a semi-regular right-angled building.
We begin with a short Section~\ref{ss:boundedsection} that uses the results of the previous sections to show
that the set of chambers fixed by a ball fixator is bounded; see Theorem~\ref{th:fixedpointset}.
In Section~\ref{ss:rootwing}, we associate a root wing group $U_\alpha$ to each root (Definition~\ref{def:rootwinggroup}),
we show that $U_\alpha$ acts transitively on the set of apartments through $\alpha$ (Proposition~\ref{pr:trans-apt})
and we adopt some facts from~\cite{CM13} to the setting of root wing groups.

We then continue towards our characterization of the open subgroups of the full automorphism group of a semi-regular locally finite right-angled building.
Our final result is Theorem~\ref{th:openpropermain} showing that every proper open subgroup is a finite index subgroup of the stabilizer of a proper residue.
We follow, to a very large extent, the strategy taken by Pierre-Emmanuel Caprace and Timoth\'ee Marquis in~\cite{CM13}
in their study of open subgroups of Kac--Moody groups over finite fields.
In particular, we show that an open subgroup of $\Aut(\Delta)$ contains sufficiently many root wing groups, and much of the subtleties of the proof
go into determining precisely the types of the root groups contained in the open subgroup;
this will, in turn, pin down the residue, the stabilizer of which contains the given open subgroup as a finite index subgroup.

In the final Section~\ref{se:appl}, we mention two applications of our main theorem.
The first is a rather immediate corollary, namely the fact that the automorphism group of a semi-regular locally finite right-angled building
is a Noetherian group; see Proposition~\ref{pr:noeth}.
The second application shows that every open subgroup of the automorphism group is the reduced envelope of a cyclic subgroup;
see Proposition~\ref{pr:env}.

\paragraph*{Acknowledgments.}

This paper would never have existed without the help of Pierre-Emmanuel Caprace.
Not only did he suggest the study of open subgroups of the automorphism group of right-angled buildings to us;
we also benefited a lot from discussions with him.

Two anonymous referees did a great job in pointing out inaccuracies and suggesting improvements for the exposition.

We also thank the Research Foundation in Flanders (F.W.O.-Vlaanderen) for their support through the project
``Automorphism groups of locally finite trees'' (G011012N).

\section{Right-angled Coxeter groups}\label{se:RAC}

We begin by recalling some basic definitions and facts about Coxeter groups.
\begin{defn}\label{def:cox}
\begin{enumerate}[(i)]
    \item
    A \emph{Coxeter group} is a group $W$ with generating set $S = \{ s_1,\dots,s_n \}$ and with presentation
    \[ W = \langle s \in S \mid (st)^{m_{st}} = 1 \rangle \]
    where $m_{ss} = 1$ for all $s \in S$ and $m_{st} = m_{ts} \geq 2$ for all $i \neq j$.
    It is allowed that $m_{st} = \infty$, in which case the relation involving $st$ is omitted.
    The pair $(W,S)$ is called a \emph{Coxeter system} of \emph{rank $n$}.
    The matrix $M = (m_{s_is_j})$ is called the \emph{Coxeter matrix} of $(W,S)$.
    The Coxeter matrix is often conveniently encoded by its \emph{Coxeter diagram}, which is a labeled graph with vertex set $S$
    where two vertices are joined by an edge labeled $m_{st}$ if and only if $m_{st} \geq 3$.
    \item
    A Coxeter system $(W,S)$ is called \emph{right-angled} if all entries of the Coxeter matrix are $1$, $2$ or $\infty$.
    In this case, we call the Coxeter diagram $\Sigma$ of $W$ a \emph{right-angled Coxeter diagram};
    all its edges have label $\infty$.
\end{enumerate}
\end{defn}

\begin{defn}\label{parb}
Let $(W,S)$ be a Coxeter system and let $J \subseteq S$.
\begin{enumerate}[(i)]
    \item
        We define $W_J := \langle s \mid s \in J \rangle \leq W$ and we call this a \emph{standard parabolic subgroup} of $W$.
        It is itself a Coxeter group, with Coxeter system $(W_J, J)$.
        Any conjugate of a standard parabolic subgroup $W_J$ is called a \emph{parabolic subgroup} of $W$.
    \item\label{it:parbspherical}
        The subset $J \subseteq S$ is called a \emph{spherical subset} if $W_J$ is finite.
        When $(W,S)$ is right-angled, $J$ is spherical if and only if $|st| \leq 2$ for all $s, t \in J$.
    \item\label{it:essential}
        The subset $J \subseteq S$ is called \emph{essential} if each irreducible component of $J$ is non-spherical.
        In general, the union $J_0$ of all irreducible non-spherical components of $J$ is called the \emph{essential component} of $J$.

        If $P$ is a parabolic subgroup of $W$ conjugate to some $W_J$, then the \emph{essential component} $P_0$ of $P$ is the corresponding conjugate of $W_{J_0}$,
        where $J_0$ is the essential component of $J$.
        Observe that $P_0$ has finite index in $P$.
    \item\label{it:parclo}
        Let $E \subseteq W$.
        We define the \emph{parabolic closure} of $E$, denoted by $\Pc(E)$, as the smallest parabolic subgroup of $W$ containing $E$.
\end{enumerate}
\end{defn}

\begin{lem}[{\cite[Lemma 2.4]{CM13}}]\label{lem2.4}
Let $H_1 \leq H_2$ be subgroups of $W$.
If $H_1$ has finite index in $H_2$, then $\Pc(H_1)$ has finite index in $\Pc(H_2)$.
\end{lem}

\subsection{A poset of reduced words}\label{ss:poset}

Let $\Sigma = (W, S)$ be a right-angled Coxeter system
and let $M_S$ be the free monoid over $S$, the elements of which we refer to as \emph{words}.
Notice that there is an obvious map $M_S \to W$ denoted by $w \mapsto \ow$;
if $w \in M_S$ is a word, then its image $\ow$ under this map is called the element \emph{represented by $w$},
and the word $w$ is called a \emph{representation of $\ow$}.
For $w_1, w_2 \in M_S$, we write $w_1 \sim w_2$ when $\overline{w_1} = \overline{w_2}$.
By some slight abuse of notation, we also say that $w_2$ is a representation of $w_1$ (rather than a representation of $\overline{w_1}$).

\begin{defn}\label{def:operationsRAB}
A $\Sigma$-\emph{elementary operation} on a word $w \in M_S$ is an operation of one of the following two types:
\begin{compactenum}[(1)]
	\item
	Delete a subword of the form $ss$, with $s\in S$.
	\item
	Replace a subword $st$ by $ts$ if $m_{st}=2$.
\end{compactenum}
A word $w \in M_S$ is called \emph{reduced} (with respect to $\Sigma$) if it cannot be shortened by a sequence of $\Sigma$-elementary operations.
\end{defn}
Clearly, applying elementary operations on a word $w$ does not alter its value in $W$.
Conversely, if $w_1 \sim w_2$ for two words $w_1, w_2 \in M_S$, then $w_1$ can be transformed into $w_2$ by a sequence of $\Sigma$-elementary operations.
The number of letters in a reduced representation of $\ow \in W$ is called the \emph{length} of $\ow$ and is denoted by $l(\ow)$.
Tits proved in~\cite{Titswords} (for arbitrary Coxeter systems) that two \emph{reduced} words represent the same element of $W$ if and only if
one can be obtained from the other by a sequence of elementary operations of type (2) (or rather its generalization to all values for $m_{st}$).

\begin{defn}
Let $w = s_1 s_2 \dotsm s_\ell \in M_S$.
If $\sigma \in \Sym(\ell)$, then we let $\sigma.w$ be the word obtained by permuting the letters in $w$ according to the permutation $\sigma$, i.e.,
\[ \sigma.w := s_{\sigma(1)} s_{\sigma(2)} \dotsm s_{\sigma(\ell)} . \]
In particular, if $w'$ is obtained from $w$ by applying an elementary operation of type (2) replacing $s_i s_{i+1}$ by $s_{i+1} s_i$, then
$\sigma.w = w'$ for $\sigma=(i \ i+1)\in \Sym(\ell)$.
In this case, $s_i$ and $s_{i+1}$ commute and we call $\sigma = (i \ i+1)$ a \emph{$w$\dash elementary transposition}.
\end{defn}

In this way, we can associate an elementary transposition to each $\Sigma$\dash elementary operation of type (2).
It follows that two reduced words $w$ and $w'$ represent the same element of $W$ if and only if
\begin{multline*}
    w' = (\sigma_k \dotsm \sigma_1).w, \text{ where each } \sigma_i \text{ is a } \\
        (\sigma_{i-1}\dotsm \sigma_1).w \text{-elementary transposition},
\end{multline*}
i.e.\@, if $w'$ is obtained from $w$ by a sequence of elementary transpositions.

\begin{defn}\label{def:rep}
If $w \in M_S$ is a reduced word of length $\ell$, then we define
\begin{multline*}
    \Rep(w) := \{ \sigma\in \Sym(\ell) \mid \sigma =\sigma_k\dotsm \sigma_1, \text{ where each } \sigma_i \text{ is a } \\
        (\sigma_{i-1}\dotsm\sigma_1).w\text{-elementary transposition}\}.
\end{multline*}
In other words, the set $\Rep(w)$ consists of the permutations of $\ell$ letters which give rise to reduced representations of $w$.
\end{defn}


We now define a partial order $\prec_w$ on the letters of a reduced word $w$ in $M_S$ with respect to $\Sigma$.
\begin{defn}[{\cite[Definition 2.6]{DMSS}}]\label{def:poset}
Let $w=s_1\dotsm s_\ell$ be a reduced word of length $\ell$ in $M_S$
and let $I_w = \{1, \dots , \ell \}$.
We define a partial order ``$\prec_w$'' on $I_w$ as follows:
\[
i \succ_w j \iff \sigma (i) < \sigma (j) \text{ for all } \sigma\in \Rep(w).
\]
\end{defn}
Note that $i \succ_w j$ implies that $i < j$.
As a mnemonic, one can regard $i \succ_w j$ as ``$i\rightarrow j$'', i.e., the generator $s_j$ comes always after the generator $s_i$ regardless of the reduced representation of $w$.

We point out a couple of basic but enlightening consequences of the definition of this partial order.
\begin{obs}\label{obsrep2}
Let $w= s_1 \dotsm s_i \dotsm s_j \dotsm s_\ell$ be a reduced word in $M_S$ with respect to a right-angled Coxeter diagram $\Sigma$.
\begin{enumerate}[(i)]
    \item\label{obsrep:3}
        If $|s_is_j|=\infty$, then $i \succ_w j$.

        The converse is not true.
        Indeed, suppose there is $i<k<j$ such that $|s_is_k|=\infty$ and $|s_ks_j|=\infty$. Then $i\succ_w j$, independently of whether $|s_is_j|=2$ or $\infty$.
    \item\label{obsrep2:4}
        If $i\not\succ_w j$, then by~\eqref{obsrep:3}, it follows that $|s_is_j|=2$
        and, moreover, for each $k\in \{i+1, \dots, j-1\}$, either $|s_is_k|=2$ or $|s_ks_j|=2$ (or both).

    \item\label{obsrep:5}
        On the other hand, if $s_j$ and $s_{j+1}$ are consecutive letters in $w$, then $|s_j s_{j+1}|=\infty$ if and only if $j \succ_{w} j+1$.
\end{enumerate}
\end{obs}


\begin{lem}[{\cite[Lemma 2.8]{DMSS}}]\label{le:reduced}
Let $w = w_1 \cdot s_i \dotsm s_j \cdot w_2 \in M_S$ be a reduced word.
If $i\not\succ_w j$, then there exist two reduced representations of $w$ of the form
\[ w_1 \dotsm s_i s_j \dotsm w_2 \quad \text{ and } \quad w_1\dotsm s_j s_i \dotsm w_2 , \]
i.e.,
the positions of $s_i$ and $s_j$ can be exchanged using only elementary operations on the generators $\{s_i, s_{i+1}, \dots, s_{j-1}, s_j\}$,
without changing the prefix $w_1$ and the suffix $w_2$.
\end{lem}

\subsection{Firm elements of right-angled Coxeter groups}\label{ss:firm}

In this section we define firm elements in a right-angled Coxeter group $W$
and we introduce the concept of firmness to measure ``how firm'' an arbitrary elements of $W$ is.
This concept will be used over and over throughout the paper.
See, in particular, Definition~\ref{def:Ai}, Theorem~\ref{th:flex}, Theorem~\ref{th:fixedpointset} and Proposition~\ref{pr:trans-apt}.
Our main result in this section is Theorem~\ref{th:boundforN}, showing that the firmness of elements cannot drop below a certain value once
they become sufficiently long.

\begin{defn}\label{def:formF}
Let $\ow \in W$ be represented by some reduced word $w = s_1\dotsm s_\ell \in M_S$.
\begin{enumerate}[(i)]
    \item\label{it:firm}
        We say that $\ow$ is \emph{firm} if $i \succ_w \ell$ for all $i \in \{ 1,\dots,\ell-1 \}$.
        In other words, $\ow$ is firm if its final letter $s_\ell$ is in the final position in each possible reduced representation of $\ow$.
        Equivalently, $\ow$ is firm if and only if there is a unique $r \in S$ such that $l(\ow r) < l(\ow)$.
    \item
        Let $F^\#(\ow)$ be the largest $k$ such that $\ow$ can be represented by a reduced word in the form
        \[ s_1 \dotsm s_k t_{k+1} \dotsm t_\ell, \text{ with } s_1 \dotsm s_k \text{ firm} . \]
        We call $F^\#(\ow)$ the \emph{firmness} of $\ow$. 
        We will also use the notation $F^\#(w) := F^\#(\ow)$. 
\end{enumerate}
\end{defn}


\begin{lem}\label{le:firmament}
Let $w=s_1\dotsm s_kt_{k+1}\dotsm t_\ell$ be a reduced word such that $s_1\dotsm s_k$ is firm and $F^\#(w)=k$.
Then
\begin{enumerate}[\rm (i)]
    \item\label{it:skti2}
        $|s_kt_i|=2$ for all $i\in\{k+1, \dots, \ell\}$.
    \item\label{it:relfirm}
        $i \succ_w k$ for all $i \in \{ 1,\dots,k-1 \}$.
    \item\label{it:Fwr}
        Let $r\in S$. If $l(\ow r)>l(\ow)$, then $F^\#(wr)\geq F^\#(w)$.
\end{enumerate}
\end{lem}

\begin{proof}
\begin{enumerate}[(i)]
    \item
        Assume the contrary and let $j$ be minimal such that $|s_kt_j|=\infty$.
        Using elementary operations to swap $t_j$ to the left in $w$ as much as possible, we rewrite
        \[ w\sim s_1\dotsm s_k t_1'\dotsm t_p't_j\dotsm \]
        as a word with $s_1\dotsm s_k t_1'\dotsm t_p't_j$ firm, which is a contradiction to the maximality of $k$.
    \item
        The fact that the prefix $p = s_1 \dotsm s_k$ is firm tells us that $i \succ_p k$ for all $i \in \{ 1,\dots,k-1 \}$.
        By Lemma~\ref{le:reduced}, this implies that also $i \succ_w k$ for all $i \in \{ 1,\dots,k-1 \}$.
    \item
        Since $l(\ow r) > l(\ow)$, firm prefixes of $w$ are also firm prefixes of $wr$, hence the result.
    \qedhere
\end{enumerate}
\end{proof}

The following definition will be a useful tool to identify which letters of the word appear in a firm subword.
\begin{defn}\label{def:Iposet}
Let $w=s_1\dotsm s_\ell \in M_S$ be a reduced word and consider the poset $(I_w, \prec_w)$ as in Definition~\ref{def:poset}.
For any $i\in \{1, \dots, \ell\}$, we define
\[
I_{w}(i)=\bigl\{ j\in \{1, \dots, \ell\} \mid j\succ_w i \bigr\}.
\]
In words, $I_w(i)$ is the set of indices $j$ such that $s_j$  comes at the left of $s_i$ in any reduced representation of the element $w\in W$.
\end{defn}

\begin{obs}\label{obs:Iposet}
Let $w=s_1\dotsm s_\ell \in M_S$ be a reduced word.
\begin{enumerate}[(i)]
    \item\label{it:Iw}
    Let $i\in \{1, \dots, \ell\}$ and write $I_w(i) = \{ j_1,\dots,j_k \}$ with $j_p < j_{p+1}$ for all~$p$.
    Then we can perform elementary operations on $w$ so that
    \[
    w\sim s_{j_1}\dotsm s_{j_k} s_i t_1 \dotsm t_q
    \]
    and the word $s_{j_1}\dotsm s_{j_k} s_i$ is firm.

    In particular, if $I_w(i)=\emptyset$, then we can rewrite $w$ as $s_iw_1$.



    \item\label{it:Ic}
    If $i\succ_w j$, then $I_w(i)\subsetneq I_w(j)$.

    \item\label{it:Fwmax}
    It follows from~\eqref{it:Iw} that $F^\#(w)=\max_{i\in \{1, \dots, \ell\}}|I_w(i)|+1$.
\end{enumerate}
\end{obs}

\begin{rem}
If the Coxeter system $(W,S)$ is spherical, then $F^\#(\ow)=1$ for all $\ow \in W$.
Indeed, as each pair of distinct generators commute, we always have $I_w(i)=\emptyset$.
\end{rem}

The next definition will allow us to deal with possibly infinite words.
\begin{defn}\label{sphericalset}
\begin{enumerate}[(i)]
    \item
    A (finite or infinite) sequence $(r_1, r_2, \dotsc)$ of letters in $S$ will be called a \emph{reduced increasing sequence}
    if $l(r_1\dotsm r_i) < l(r_1\dotsm r_ir_{i+1})$ for all $i \geq 1$.
    \item
    Let $w \in M_S$.
    A sequence $(r_1, r_2, \dotsc)$ of letters in $S$ will be called a \emph{reduced increasing $w$-sequence}
    if $l(w r_1\dotsm r_i) < l(w r_1\dotsm r_ir_{i+1})$ for all $i \geq 0$.
\end{enumerate}
\end{defn}

\begin{lem}\label{lem:auxfinite}
Let  $\alpha = (r_1, r_2, \dotsc)$ be a reduced increasing sequence in $S$.
Assume that each subsequence of $\alpha$ of the form
\[
(r_{a_1}, r_{a_2}, \dots ) \text{ with }|r_{a_i}r_{a_{i+1}}|=\infty \text{ for all } i\]
has $\leq b$ elements.
Then there is some positive integer $f(b)$ depending only on $b$ and on the Coxeter system $(W,S)$, such that $\alpha$ has $\leq f(b)$ elements.
\end{lem}
\begin{proof}
We will prove this result by induction on $|S|$; the case $|S| = 1$ is trivial.

Suppose now that $|S| \geq 2$.
If $(W,S)$ is a spherical Coxeter group, then the result is obvious since the length of any reduced increasing sequence is bounded by the length of the longest element of $W$.
We may thus assume that there is some $s\in S$ that does not commute with some other generator in $S\setminus \{s\}$.

Since the sequence $\alpha$ is a reduced increasing sequence, we know that between any two $s$'s, there must be some $t_i$ such that $|st_i|=\infty$.
Consider the subsequence of $\alpha$ given by
\[ (s, t_1, s, t_2, \dotsc ).  \]
This subsequence has $\leq b$ elements by assumption, and between any two generators $s$ in the original sequence $\alpha$, we only use letters in $S\setminus\{s\}$.
The result now follows from the induction hypothesis.
\end{proof}

\begin{lem}\label{le:ris-bound}
Let $\ow \in W$.
Then there is some $k(\ow)\in\N$, depending only on $\ow$, such that for every reduced increasing $w$-sequence $(r_1, r_2,\dotsc)$ in $S$, we have
\[ F^\#(wr_1\dotsm r_{k(\ow)}) > F^\#(w) . \]
\end{lem}

\begin{proof}
Assume that there is a reduced increasing $w$-sequence $\alpha = (r_1, r_2, \dotsc)$ in $S$ such that
\begin{equation}\label{eq:Feq}
    F^\#(wr_1 \dotsm r_i) = F^\#(w) \text{ for all } i . \tag{$*$}
\end{equation}
Let $w_0=w$, $w_i=w_{i-1}r_i$ and denote $I_i = I_{w_i}(l(\overline{w})+i)$ for all $i$.
Let $b=F^\#(w)$. By assumption~\eqref{eq:Feq} and Observation~\ref{obs:Iposet}\eqref{it:Fwmax}, we have $|I_i|\leq b-1$ for all $i$.
Moreover, if $i<j$ with $|r_ir_j|=\infty$, then $I_i\subsetneq I_j$ by Observations~\ref{obsrep2}\eqref{obsrep:3} and~\ref{obs:Iposet}\eqref{it:Ic};
it follows that each subsequence of $\alpha$ of the form
\[
(r_{a_1}, r_{a_2}, \dots ) \text{ with }|r_{a_i}r_{a_{i+1}}|=\infty \text{ for all } i\]
has at most $b$ elements.
By Lemma~\ref{lem:auxfinite}, this implies that the sequence $\alpha$ has at most $f(b)$ elements.
We conclude that every reduced increasing $w$\dash sequence $(r_1, r_2,\dotsc, r_{k(\ow)})$ in $S$ with $k(\ow) := f(F^\#(w)) + 1$
must have strictly increasing firmness.
\end{proof}

\begin{thm}\label{th:boundforN}
Let $(W,S)$ be a right-angled Coxeter system.
For all $n\geq 0$, there is some $d(n) \in \N$ depending only on $n$, such that
$F^\#(\ow)>n$ for all $\ow \in W$ with $l(\ow)>d(n)$.
\end{thm}
\begin{proof}
This follows by induction on $n$ from Lemma~\ref{le:ris-bound} since there are only finitely many elements in $W$ of any given length.
%
%
\end{proof}

\section{Right-angled buildings}\label{se:RAB}

We will start by recalling the procedure of ``closing squares'' in right-angled buildings from~\cite{DMSS} and we define the square closure of a set of chambers.
Our goal in this section is to describe the square closure of a ball in the building
and to show that this is a bounded set, i.e., it has finite diameter; see Theorem~\ref{th:flex}.

\subsection{Preliminaries}\label{ss:RAB-pre}

We regard buildings as chamber systems, following the notation in \cite{Weiss}.
We briefly recall the basic notions and refer the reader to \textit{loc.\@~cit.\@} for more details.
Recall the notation from section~\ref{ss:poset}.

\begin{defn}
\begin{enumerate}[(i)]
\item 
    Let $\Delta$ be an edge-colored graph with color set $S$.
    Let $J \subseteq S$.
    A \textit{$J$-residue} of $\Delta$ is a connected component of the subgraph of $\Delta$ obtained from $\Delta$ by discarding all the edges whose color is not in $J$.
    A residue of $\Delta$ is a $J$-residue for some $J \subseteq S$.
    If $s \in S$ then an $\{s\}$-residue is called an \textit{$s$-panel}.
\item
    A chamber system is an edge-colored graph $\Delta$ with color set $S$ such that for each $s \in S$, all $s$-panels of $\Delta$ are complete graphs with at least two vertices. We will refer to the vertices of $\Delta$ as \textit{chambers} and we will denote the vertex set by $\ch(\Delta)$.
    The cardinality of $S$ is called the \textit{rank} of $\Delta$.
\item
    Two chambers $c_1, c_2$ are called \textit{$s$-adjacent} if they are connected by an edge with color $s$, and we denote this by $c_1 \adj{s} c_2$.

\item
    A chamber system is \textit{thin} if every panel contains exactly two chambers and is \textit{thick} if every panel contains at least three chambers.
\item
    A \textit{gallery} in a chamber system $\Delta$ is a sequence $(v_0, v_1, \dots, v_k)$ of chambers such that $v_{i-1}$ is adjacent to $v_{i}$ for all $i$.
    We then call this a gallery \textit{from $v_0$ to $v_k$}; the number $k$ is the \textit{length} of the gallery.
    If for each $i$, $v_{i-1}$ is $s_i$\dash adjacent to $v_{i}$, then the word $w = s_1 s_2 \dotsm s_k \in M_S$ is called the \textit{type} of the gallery.
\end{enumerate}
\end{defn}

\begin{defn}
Let $(W, S)$ be a Coxeter system.
A \textit{building of type $(W, S)$} is a pair~$(\Delta, \delta)$, where
$\Delta$ is a chamber system with index set $S$ and
$\delta$ is a map
\[ \delta \colon \ch(\Delta) \times \ch(\Delta) \to W \]
such that for each reduced word $w \in M_S$
and for each pair of chambers $c_1, c_2 \in \ch(\Delta)$, we have
\[ \delta(c_1, c_2) = \overline{w} \iff \text{there is a gallery in $\Delta$ of type $w$ from $c_1$ to $c_2$.} \]
We call the group $W$ the \textit{Weyl group} and the map $\delta$ the \textit{Weyl distance}.    
\end{defn}

\begin{rem}
    Notice that with our combinatorial setup, the basic objects are chambers, and panels contain chambers.
    There exist various other realizations of buildings (that are nevertheless equivalent) in which the containment is the other way around.
    We refer, for instance, to the introduction of \cite{AbramenkoBrown} for a good overview.    
\end{rem}

\begin{defn}\label{def:roots}
    Let $\Sigma = (W,S)$ be a Coxeter system.
    \begin{enumerate}[(i)]
        \item\label{it:C1}
            We define a thin building $\Delta_\Sigma$ of type $(W, S)$ by taking $\ch(\Delta_\Sigma) = W$ as the set of chambers,
            declaring $x \adj{s} y$ for $s \in S$ if and only if $x^{-1}y=s$,
            and defining a Weyl distance $\delta(x,y) := $ for all $x,y \in \ch(\Delta_\Sigma)$.
        \item
            Let $\Delta$ be an arbitrary building of type $(W,S)$.
            An \textit{apartment} in $\Delta$ is a subbuilding of $\Delta$ that is $\delta$-isometric to $\Delta_\Sigma$.
        \item
            Let $\Delta_\Sigma$ be as in \eqref{it:C1}.
            A \textit{reflection} of $\Delta_\Sigma$ is a non-trivial element $r \in W$ fixing an edge (i.e., a panel) of $\Delta_\Sigma$;
            such an element $r$ is always an involution of $W$.
            The set of edges fixed by $r$ is called the \textit{wall} of $r$.
            
            To each reflection $r$, we can associate a partition of the chamber set into two parts, as follows.
            Let $\{ x, y \}$ be a panel in the wall of $r$.
            Then each chamber of $\Delta_\Sigma$ is either nearer to $x$ than to $y$ or nearer to $y$ than to $x$,
            so we get two parts
            $\{ c \in W \mid \dist(c, x) < \dist(c, y) \}$
            and its complement
            $\{ c \in W \mid \dist(c, y) < \dist(c, x) \}$.
            These two parts are called the \textit{roots} associated to $r$ and they are interchanged by $r$.
            (They are independent of the choice of $\{ x,y \}$ in the wall of $r$. See \cite[Proposition~3.11]{Weiss}.)
            In particular, if $\alpha$ is a root, then its complement is again a root and is denoted by $-\alpha$.
        \item
            Let $\Delta$ be an arbitrary building of type $(W,S)$.
            A \textit{root} of $\Delta$ is then defined to be a root in one of its apartments.
            (Roots are also sometimes referred to as \textit{half-apartments}.)
    \end{enumerate}
\end{defn}

From now on, let $(W,S)$ be a right-angled Coxeter system with Coxeter diagram $\Sigma$
and let $\Delta$ be a right-angled building of type $(W,S)$.

\begin{defn}\label{galdist}
\begin{enumerate}[(i)]
    \item 
        Let $\delta \colon \Delta \times \Delta \to W$ be the Weyl distance of the building~$\Delta$.
        The \emph{gallery distance} between the chambers $c_1$ and $c_2$ is defined as
        \[ \dw(c_1, c_2) := l(\delta(c_1,c_2)) , \]
        i.e., the length of a minimal gallery between the chambers $c_1$ and $c_2$.
    \item 
        For a fixed chamber $c_0\in\ch(\Delta)$ we define the \emph{spheres} at a fixed gallery distance from  $c_0$ as
        \[ \Sp(c_0, n) := \{ c\in \ch(\Delta) \mid \dw(c_0, c) = n\} \]
        and the \emph{balls} as
        \[ \B(c_0, n) := \{ c\in \ch(\Delta) \mid \dw(c_0, c) \leq n\}. \]
\end{enumerate}
\end{defn}

\begin{defn}\phantomsection\label{def:par}
\begin{enumerate}[(i)]
    \item
        Let $c$ be a chamber in $\Delta$ and $\Rd$ be a residue in $\Delta$.
        The \emph{projection} of $c$ on $\Rd$ is the unique chamber in $\Rd$ that is closest to $c$ and it is denoted by $\proj_{\Rd}(c)$.
    \item
        If $\Rd_1$ and $\Rd_2$ are two residues, then the set of chambers
        \[ \proj_{\Rd_1}(\Rd_2) := \{ \proj_{\Rd_1}(c) \mid c\in \ch(\Rd_2) \} \]
        is again a residue and the rank of $\proj_{\Rd_1}(\Rd_2)$ is bounded above by the ranks of both $\Rd_1$ and $\Rd_2$; see~\cite[Section 2]{Caprace}.
    \item\label{it:par}
        The residues $\Rd_1$ and $\Rd_2$ are called \emph{parallel} if \mbox{$\proj_{\Rd_1}(\Rd_2)=\Rd_1$} and $\proj_{\Rd_2}(\Rd_1)=\Rd_2$.
\end{enumerate}
\end{defn}

In particular, if $\Pa_1$ and $\Pa_2$ are two parallel panels, then the chamber sets of $\Pa_1$ and $\Pa_2$ are mutually in bijection under the respective projection maps (see again~\cite[Section 2]{Caprace}).

\begin{defn}\label{perp}
Let $J\subseteq S$. We define the set
\[ J^\perp = \{ t\in S\setminus J \mid ts=st \text{ for all } s\in J \}. \]
If $J=\{s\}$, then we write the set $J^\perp$ as $s^\perp$.
\end{defn}

\begin{prop}[{\cite[Proposition 2.8]{Caprace}}]\label{Proposition2.8}
Let $\Delta$ be a right-angled building of type $(W, S)$.
\begin{enumerate}[\rm (i)]
\item Any two parallel residues have the same type.
\item Let $J\subseteq S$. Given a residue $\Rd$ of type $J$, a residue $\Rd'$ is parallel to $\Rd$ if and only if $\Rd'$ is of type $J$, and $\Rd$ and $\Rd'$ are both contained in a common residue of type $J\cup J^\perp$.
\end{enumerate}
\end{prop}

\begin{prop}[{\cite[Corollary 2.9]{Caprace}}]
Let $\Delta$ be a right-angled building. Parallelism of residues of $\Delta$ is an equivalence relation.
\end{prop}


Another very important notion in right-angled buildings is that of \emph{wings}, introduced in \cite[Section 3]{Caprace}.
For our purposes, it will be sufficient to consider wings with respect to panels.

\begin{defn}\label{def:wings}
Let $c\in \ch(\Delta)$ and $s\in S$.
Denote the unique $s$-panel containing $c$ by $\Pa_{s,c}$.
Then the set of chambers
\[ X_s(c)=\{x\in \ch(\Delta) \mid \proj_{\Pa_{s,c}}(x)=c\} \]
is called the \emph{$s$-wing} of $c$.
\end{defn}

Notice that if $\Pa$ is any $s$-panel, then the set of $s$-wings of each of the different chambers of $\Pa$ forms a partition of $\ch(\Delta)$
into equally many combinatorially convex subsets (see \cite[Proposition 3.2]{Caprace}).

\subsection{Sets of chambers closed under squares}\label{ss:squares}

We start by presenting two results proved in~\cite[Lemmas 2.9 and~2.10]{DMSS} that can be used in right-angled buildings to modify minimal galleries using the commutation relations of the Coxeter group.
We will refer to these results as the ``Closing Squares Lemmas'' (see also Figure~\ref{fig:CS} below).

\begin{lem}[Closing Squares 1]\label{Squares}
Let $c_0$ be a fixed chamber in $\Delta$.
Let $c_1, c_2 \in \Sp(c_0, n)$ and $c_3 \in \Sp(c_0, n+1)$ such that
\[
c_1 \adj{t} c_3 \quad \text{and} \quad c_2 \adj{s} c_3
\]
for some $s \neq t$.
Then $|st|=2$ in $\Sigma$ and there exists $c_4\in \Sp(c_0, n-1)$ such that
\[ c_1 \adj{s} c_4 \quad \text{and} \quad c_2 \adj{t} c_4. \]
\end{lem}

\begin{lem}[Closing Squares 2]\label{Squares2}
Let $c_0$ be a fixed chamber in $\Delta$.
Let $c_1, c_2 \in \Sp(c_0, n)$ and $c_3\in \Sp(c_0, n-1)$ such that
\[ c_1 \adj{s} c_2 \quad \text{and} \quad c_2 \adj{t} c_3 \]
for some $s \neq t$.
Then $|st|=2$ in $\Sigma$ and there exists $c_4\in \Sp(c_0, n-1)$ such that
\[ c_1 \adj{t} c_4 \quad \text{and} \quad c_3 \adj{s} c_4. \]
\end{lem}

\begin{figure}[ht!]
\centering
\begin{subfigure}[b]{0.3\textwidth}
\scalebox{0.7}{
\definecolor{ttqqqq}{rgb}{0.2,0.,0.}
\begin{tikzpicture}[scale=0.8,line cap=round,line join=round,>=triangle 45,x=1.0cm,y=1.0cm]
\clip(11.5,-2) rectangle (18,4.28);
\draw [shift={(14.0370247934,-10.3016528926)}] plot[domain=1.02538367692:2.14734467758,variable=\t]({1.*12.4897881927*cos(\t r)+0.*12.4897881927*sin(\t r)},{0.*12.4897881927*cos(\t r)+1.*12.4897881927*sin(\t r)});
\draw [shift={(14.0370247934,-10.3016528926)}] plot[domain=0.997988156797:2.15611076477,variable=\t]({1.*14.1808741908*cos(\t r)+0.*14.1808741908*sin(\t r)},{0.*14.1808741908*cos(\t r)+1.*14.1808741908*sin(\t r)});
\draw [shift={(14.0370247934,-10.3016528926)}] plot[domain=1.06235112263:2.14028252827,variable=\t]({1.*10.7291062881*cos(\t r)+0.*10.7291062881*sin(\t r)},{0.*10.7291062881*cos(\t r)+1.*10.7291062881*sin(\t r)});
\draw (16.9,2.35394050274) node[anchor=north west] {$n$};
\draw (16.8,0.55) node[anchor=north west] {$n-1$};
\draw (16.8,4.13212405444) node[anchor=north west] {$n+1$};
\draw (15.5471274469,2.6) node[anchor=north west] {$c_2$};
\draw (11.8,2.65) node[anchor=north west] {$c_1$};
\draw (14.1,-1.3) node[anchor=north west] {$c_0$};
\draw (13.71,-0.17) node[anchor=north west] {$\vdots$};
\draw (13.6247668505,0.3) node[anchor=north west] {$c_4$};
\draw (14.0092389698,4.4) node[anchor=north west] {$c_3$};
\draw [line width=1.2pt] (13.9701176712,3.90429032555)-- (12.3332286877,2.07137777611);
\draw [line width=1.2pt] (13.9701176712,3.90429032555)-- (15.5781703153,2.09268773443);
\draw [line width=1.2pt,dash pattern=on 2pt off 2pt] (12.3332286877,2.07137777611)-- (13.9068569939,0.377507567897);
\draw [line width=1.2pt,dash pattern=on 2pt off 2pt] (15.5781703153,2.09268773443)-- (13.9068569939,0.377507567897);
\draw (12.6436849104,3.4) node[anchor=north west] {$t$};
\draw (14.8582815665,3.3) node[anchor=north west] {$s$};
\draw (12.55,1.45) node[anchor=north west] {$s$};
\draw (14.9223602531,1.52) node[anchor=north west] {$t$};
\begin{scriptsize}
\draw [fill=black] (12.3332286877,2.07137777611) circle (3.5pt);
\draw [fill=black] (15.5781703153,2.09268773443) circle (3.5pt);
\draw [fill=black] (13.9,-1.6) circle (3.5pt);
\draw [fill=black] (13.9068569939,0.377507567897) circle (3.5pt);
\draw [fill=black] (13.9701176712,3.90429032555) circle (3.5pt);
\end{scriptsize}
\end{tikzpicture}
}
\subcaption{\centering Lemma~\ref{Squares}}\label{fig:a}
\end{subfigure}
\hspace{2.5cm}
\begin{subfigure}[b]{0.3\textwidth}
\definecolor{ttqqqq}{rgb}{0.2,0.,0.}
\scalebox{0.7}{
\begin{tikzpicture}[scale=0.91,line cap=round,line join=round,>=triangle 45,x=1.0cm,y=1.0cm]
\clip(11.7,-2.4) rectangle (18.5,2.8);
\draw [shift={(14.0370247934,-10.3016528926)}] plot[domain=1.02538367692:2.14734467758,variable=\t]({1.*12.4897881927*cos(\t r)+0.*12.4897881927*sin(\t r)},{0.*12.4897881927*cos(\t r)+1.*12.4897881927*sin(\t r)});
\draw [shift={(14.0370247934,-10.3016528926)}] plot[domain=1.06235112263:2.14028252827,variable=\t]({1.*10.7291062881*cos(\t r)+0.*10.7291062881*sin(\t r)},{0.*10.7291062881*cos(\t r)+1.*10.7291062881*sin(\t r)});
\draw (17.1,2.2426021271) node[anchor=north west] {$n$};
\draw (17,0.453792447377) node[anchor=north west] {$n-1$};
\draw (14.9,2.7) node[anchor=north west] {$c_2$};
\draw (11.75,2.7) node[anchor=north west] {$c_1$};
\draw (14.5,-1.1) node[anchor=north west] {$c_0$};
\draw (14.2,-0.15) node[anchor=north west] {$\vdots$};
\draw (12.7186231146,0.31) node[anchor=north west] {$c_4$};
\draw (15.8,0.26) node[anchor=north west] {$c_3$};
\draw [line width=1.2pt,dash pattern=on 2pt off 2pt] (12.3332286877,2.07137777611)-- (13.1886752932,0.375450417608);
\draw (12.3,1.5) node[anchor=north west] {$t$};
\draw (14.2,0.3) node[anchor=north west] {$s$};
\draw (13.4,2.65) node[anchor=north west] {$s$};
\draw (15.35,1.6) node[anchor=north west] {$t$};
\draw [line width=1.2pt] (12.3332286877,2.07137777611)-- (14.8607567251,2.16094213425);
\draw [line width=1.2pt] (14.8607567251,2.16094213425)-- (15.7609052706,0.310165392801);
\draw [line width=1.2pt,dash pattern=on 2pt off 2pt] (15.7609052706,0.310165392801)-- (13.1886752932,0.375450417608);
\begin{scriptsize}
\draw [fill=black] (12.3332286877,2.07137777611) circle (3.5pt);
\draw [fill=black] (14.8607567251,2.16094213425) circle (3.5pt);
\draw [fill=black] (14.35,-1.25) circle (3.5pt);
\draw [fill=black] (13.1886752932,0.375450417608) circle (3.5pt);
\draw [fill=black] (15.7609052706,0.310165392801) circle (3.5pt);
\end{scriptsize}
\end{tikzpicture}
}
\subcaption{Lemma~\ref{Squares2}}\label{fig:b}
\end{subfigure}
\caption{Closing Squares Lemmas}\label{fig:CS}
\end{figure}
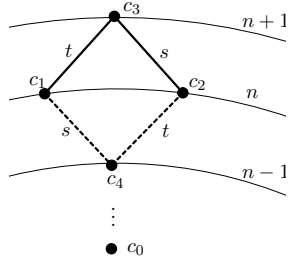
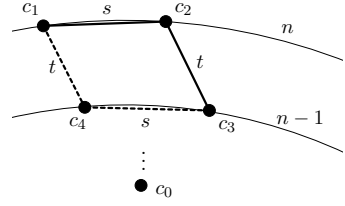

\begin{defn}\label{def:Ai}
Let $c_0$ be a fixed chamber of $\Delta$ and let $n\in \N$.
\begin{enumerate}[(i)]
    \item
        Let $c \in \ch(\Delta)$.
        Then we call $c$ \emph{firm with respect to $c_0$} if and only if $\delta(c_0, c) \in W$ is firm (as in Definition~\ref{def:formF}\eqref{it:firm}).
    \item
        We will create a partition of the sphere $\Sp(c_0, n)$ by defining
        \begin{align*}
            A_1(n) &= \{ c \in \Sp(c_0, n) \mid c \text{ is firm} \} , \\
            A_2(n) &= \{ c \in \Sp(c_0, n) \mid c \text{ is not firm} \} ,
        \end{align*}
        as in Figure~\ref{fig:S(n)}.
        Notice that this is equivalent to the definition given in \cite[Definition 4.3]{DMSS}.
        \begin{figure}[ht!]
            \captionsetup[subfigure]{justification=centering}
            \begin{subfigure}[b]{0.41\textwidth}
            \begin{tikzpicture}[scale =0.67, line cap=round,line join=round,>=triangle 45,x=1.0cm,y=1.0cm]
            \clip(3.1,2.2) rectangle (11.2,9.5);
            \draw [shift={(7.196048227047175,-8.159740084379619)}] plot[domain=0.90299902219246:2.132950643098445,variable=\t]({1.0*12.987535105233386*cos(\t r)+-0.0*12.987535105233386*sin(\t r)},{0.0*12.987535105233386*cos(\t r)+1.0*12.987535105233386*sin(\t r)});
            \draw [shift={(7.196048227047175,-8.159740084379619)}] plot[domain=0.8405516372429191:2.200323608157614,variable=\t]({1.0*14.924888309352937*cos(\t r)+-0.0*14.924888309352937*sin(\t r)},{0.0*14.924888309352937*cos(\t r)+1.0*14.924888309352937*sin(\t r)});
            \draw [line width=1.2pt](6.794841094694121,4.821596555066206)-- (6.460693486180265,6.747021618007139);
            \draw [line width=1.2pt](6.794841094694121,4.821596555066206)-- (7.939261914182559,6.746631857240333);
            \draw (6.583558053137442,4.472276888761895) node[anchor=north west] {$\vdots$};
            \draw (8.03581479852291,6.7983491334894675) node[anchor=north west] {$c_3$};
            \draw (6.571250792583328,6.847578175705924) node[anchor=north west] {$c_2$};
            \draw (5.414368300496599,6.7983491334894675) node[anchor=north west] {$c_1$};
            \draw (6.75,8.11522601277968) node[anchor=north west] {$t$};
            \draw (5.65,8.08) node[anchor=north west] {$t$};
            \draw (8.946552079527356,7.9429243650220815) node[anchor=north west] {$t$};
            \draw (5.5,5.862997331376793) node[anchor=north west] {$s$};
            \draw (6.1,6.022991718580276) node[anchor=north west] {$s$};
            \draw (6.8,5.95) node[anchor=north west] {$s$};
            \draw (6.878932306436181,2.9) node[anchor=north west] {$c_0$};
            \draw [shift={(7.196048227047175,-8.159740084379619)}] plot[domain=0.752893986406544:2.2097710988591883,variable=\t]({1.0*16.732763310877026*cos(\t r)+-0.0*16.732763310877026*sin(\t r)},{0.0*16.732763310877026*cos(\t r)+1.0*16.732763310877026*sin(\t r)});
            \draw [line width=1.2pt](5.279356698174269,6.641563393010703)-- (6.794841094694121,4.821596555066206);
            \draw [line width=1.2pt](3.4250688710404686,8.142562903677115)-- (5.279356698174269,6.641563393010703);
            \draw [line width=1.2pt](4.345950880774987,8.328507643459885)-- (5.279356698174269,6.641563393010703);
            \draw [line width=1.2pt](5.887125218266153,8.521749314799218)-- (6.460693486180265,6.747021618007139);
            \draw [line width=1.2pt](6.845058071938161,8.569341597794358)-- (6.460693486180265,6.747021618007139);
            \draw [line width=1.2pt](8.312548470057324,8.53573221277899)-- (7.939261914182559,6.746631857240333);
            \draw [line width=1.2pt](9.391698372060409,8.42834263962847)-- (7.939261914182559,6.746631857240333);
            \draw (7.7,8) node[anchor=north west] {$t$};
            \draw (4.7,8) node[anchor=north west] {$t$};
            \draw (3.5,7.8) node[anchor=north west] {$t$};
            \draw (9.783445797207118,7.056801605125864) node[anchor=north west] {$n$};
            \draw (9.598836888895406,5.173790740346401) node[anchor=north west] {$n-1$};
            \draw (9.7,9.1) node[anchor=north west] {$n+1$};
            \begin{scriptsize}
            \draw [fill=black] (7.196048227047175,-8.159740084379619) circle (3.5pt);
            \draw [fill=black] (5.279356698174269,6.641563393010703) circle (3.5pt);
            \draw [fill=black] (6.460693486180265,6.747021618007139) circle (3.5pt);
            \draw [fill=black] (7.939261914182559,6.746631857240333) circle (3.5pt);
            \draw [fill=black] (6.794841094694121,4.821596555066206) circle (3.5pt);
            \draw [fill=black] (6.9191121815025936,2.9434084491899646) circle (3.5pt);
            \draw [fill=black] (19.406165723102667,3.281343725747452) circle (3.5pt);
            \draw [fill=black] (-0.4659084580477993,2.1531227399789077) circle (3.5pt);
            \draw [fill=black] (3.4250688710404686,8.142562903677115) circle (3.5pt);
            \draw [fill=black] (4.345950880774987,8.328507643459885) circle (3.5pt);
            \draw [fill=black] (5.887125218266153,8.521749314799218) circle (3.5pt);
            \draw [fill=black] (6.845058071938161,8.569341597794358) circle (3.5pt);
            \draw [fill=black] (8.312548470057324,8.53573221277899) circle (3.5pt);
            \draw [fill=black] (9.391698372060409,8.42834263962847) circle (3.5pt);
            \end{scriptsize}
            \end{tikzpicture}
            \subcaption{$c_i$ firm: for all $t \neq s$, $l(\delta(c_0,c_i)t) > l(\delta(c_0, c_i))$.}
            \end{subfigure}
            \hspace{1.5cm}
            \begin{subfigure}[b]{0.41\textwidth}
            \begin{tikzpicture}[scale=0.70, line cap=round,line join=round,>=triangle 45,x=1.0cm,y=1.0cm]
            \clip(3.7,1.2) rectangle (10.3,8);
            \draw [shift={(7.5,-7.66)}] plot[domain=0.8281072191624526:2.264578884900933,variable=\t]({1.0*11.592980634849695*cos(\t r)+-0.0*11.592980634849695*sin(\t r)},{0.0*11.592980634849695*cos(\t r)+1.0*11.592980634849695*sin(\t r)});
            \draw [shift={(7.5,-7.66)}] plot[domain=0.8158107657899643:2.23901068745554,variable=\t]({1.0*14.417503251256784*cos(\t r)+-0.0*14.417503251256784*sin(\t r)},{0.0*14.417503251256784*cos(\t r)+1.0*14.417503251256784*sin(\t r)});
            \draw [line width=1.2pt](4.877468301282304,6.516978785665895)-- (4.13142200772546,3.432785146660134);
            \draw [line width=1.2pt](4.877468301282304,6.516978785665895)-- (5.810790384722948,3.8092532832637183);
            \draw [line width=1.2pt](5.810790384722948,3.8092532832637183)-- (6.622348908479946,6.730765391790445);
            \draw [line width=1.2pt](5.810790384722948,3.8092532832637183)-- (8.406499532206302,6.7289769823330285);
            \draw [line width=1.2pt](6.622348908479946,6.730765391790445)-- (7.510742996295576,3.9329756571826984);
            \draw [line width=1.2pt](8.406499532206302,6.7289769823330285)-- (8.811721489227729,3.858532316866512);
            \draw [line width=1.2pt](6.484130852902304,3.5) node[anchor=north west] {$\vdots$ };
            \draw (4.7,7.2) node[anchor=north west] {$c_1$};
            \draw[color=black] (6.76802401325785,7.026868238155123) node {$c_2$};
            \draw[color=black] (8.558734717038993,7.026868238155123) node {$c_3$};
            \draw[color=black] (4.3,5.1) node {$t$};
            \draw[color=black] (5.064665051124568,5.2361575343739775) node {$s$};
            \draw[color=black] (6.05,5.5) node {$s$};
            \draw[color=black] (9.6,6.9) node {$n$};
            \draw[color=black] (9.7,4.1) node {$n-1$};
            \draw[color=black] (7.8,5.75) node {$s$};
            \draw[color=black] (6.8,5.6) node {$t$};
            \draw[color=black] (8.35,5.5) node {$t$};
            \draw[color=black] (7.1,1.5) node {$c_0$};
            \begin{scriptsize}
            \draw [fill=black] (4.877468301282304,6.516978785665895) circle (3.5pt);
            \draw [fill=black] (6.622348908479946,6.730765391790445) circle (3.5pt);
            \draw [fill=black] (8.406499532206302,6.7289769823330285) circle (3.5pt);
            \draw [fill=black] (4.13142200772546,3.432785146660134) circle (3.5pt);
            \draw [fill=black] (5.810790384722948,3.8092532832637183) circle (3.5pt);
            \draw [fill=black] (8.811721489227729,3.858532316866512) circle (3.5pt);
            \draw [fill=black] (7.510742996295576,3.9329756571826984) circle (3.5pt);
            \draw [fill=black] (6.75,1.8) circle (3.5pt);
            \end{scriptsize}
            \end{tikzpicture}
            \subcaption{$c_i$ not firm: for some $t \neq s$, $l(\delta(c_0, c_i)t) < l(\delta(c_0, c_i))$.}
            \end{subfigure}
            \caption{Partition of $\Sp(c_0, n)$.}\label{fig:S(n)}
        \end{figure}
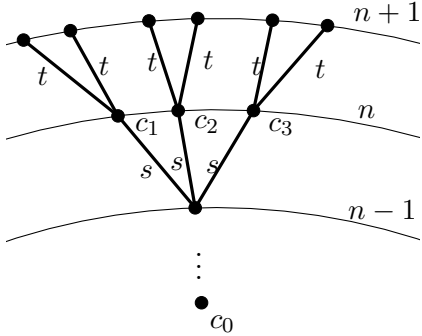
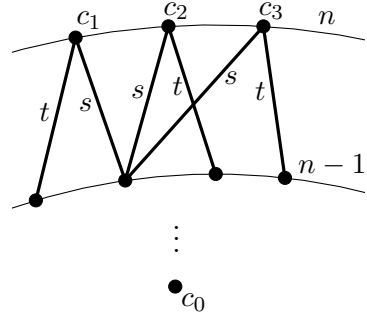
\item
    Let $c\in \Sp(c_0, k)$ for some $k> n$.
    We say that $c$ is \emph{$n$-flexible with respect to $c_0$} if
    for each minimal gallery $\gamma = (c_0, c_1, \dots, c_{n+1}, \dots, c_k = c)$ from $c_0$ to $c$,
    none of the chambers $c_{n+1},\dots,c_k$ is firm.
    By convention, all chambers of $\B(c_0, n)$ are also $n$-flexible with respect to~$c_0$.

    Observe that a chamber $c$ is $n$-flexible with respect to $c_0$ if and only if $F^\#(\delta(c_0, c)) \leq n$.
    In particular, if $c$ is $n$-flexible, then so is any chamber on any minimal gallery between $c_0$ and $c$.
\item
    We define the \emph{$n$-flex of $c_0$}, denoted by $\Flex(c_0, n)$,
    to be the set of all chambers of $\Delta$ that are $n$-flexible with respect to $c_0$.
\end{enumerate}
\end{defn}

We also record the following result, which we rephrased in terms of firm chambers;
its Corollary~\ref{co:ballwing} will be used several times in Section~\ref{se:AutRAB}.

\begin{lem}[{\cite[Lemma 2.15]{DMSS}}]\label{lemmaA}
    Let $c_0$ be a fixed chamber of $\Delta$ and let $s\in S$.
    Let $d \in \Sp(c_0, n)$ and $e \in \B(c_0, n+1)\setminus \ch(\Pa_{s,d})$.
    If $c := \proj_{\Pa_{s,d}}(e) \in \Sp(c_0, n+1)$, then $c$ is not firm with respect to $c_0$.
\end{lem}
\begin{coro}\label{co:ballwing}
    Let $c_0 \in \ch(\Delta)$ and $c \in \Sp(c_0, n+1)$ such that $c$ is firm with respect to $c_0$.
    Let $d$ be the unique chamber of $\Sp(c_0, n)$ adjacent to $c$ and let $s = \delta(d,c) \in S$.
    Then $\B(c_0, n) \subset X_s(d)$.
\end{coro}
\begin{proof}
   Let $e \in \B(c_0, n)$. If $e = d$, then of course $e \in X_s(d)$, so assume $e \neq d$; then $e \in \B(c_0, n+1)\setminus \ch(\Pa_{s,d})$.
   Notice that all chambers of $\Pa_{s,d} \setminus \{ d \}$ have the same Weyl distance from $c_0$ as $c$ and hence are firm.
   By Lemma~\ref{lemmaA}, this implies that the projection of $e$ on $\Pa_{s,d}$ must be equal to $d$,
   so by definition of the $s$-wing $X_s(d)$, we get $e \in X_s(d)$.
\end{proof}

We now come to the concept of the square closure of a set of chambers of $\Delta$.
\begin{defn}\phantomsection\label{def:squareclosure}
\begin{enumerate}[(i)]
    \item We say that a subset $T\subseteq W$ is \emph{closed under squares} if the following holds:
        \begin{quote}
            If $ws_i$ and $ws_j$ are contained in $T$ for some $w\in T$ with $|s_is_j|=2$, $s_i\neq s_j$ and $ l(ws_i)=l(ws_j)=l(w)+1$,
            then also $ws_is_j=ws_js_i$ is an element of $T$.
        \end{quote}
    \item\label{it:sqcl}
        Let $c_0$ be a fixed chamber of $\Delta$.
        A set of chambers $\mathscr{C} \subseteq\ch(\Delta)$ is \emph{closed under squares} with respect to $c_0$ if for each $n\in \N$, the following holds
        (see Figure~\ref{fig:a}):
        \begin{quote}
            If $c_1, c_2 \in \mathscr{C} \cap \Sp(c_0, n)$ and $c_4\in \mathscr{C}\cap \Sp(c_0, n-1)$
            such that $c_4\adj{s_i} c_1$ and $c_4\adj{s_j}c_2$ for some $|s_is_j|=2$ with $s_i\neq s_j$,
            then the unique chamber $c_3\in \Sp(c_0, n+1)$ such that $c_3\adj{s_j}c_1$ and $c_3\adj{s_i}c_2$ is also in $\C$.
        \end{quote}
        In particular, if $\C$ is closed under squares with respect to $c_0$, then the set of Weyl distances
        $\{ \delta(c_0, c)\mid c\in \mathscr{C}\}\subseteq W$ is closed under squares.
    \item
        Let $c_0\in \ch(\Delta)$ and let $\mathscr {C}\subseteq \ch(\Delta)$.
        We define the \emph{square closure} of~$\C$ with respect to $c_0$
        to be the smallest subset of $\ch(\Delta)$ containing $\C$ and closed under squares with respect to $c_0$.
\end{enumerate}
\end{defn}

\begin{thm}\label{th:flex}
Let $c_0 \in \ch(\Delta)$ and let $n \in \N$.
The square closure of $\B(c_0, n)$ with respect to $c_0$ is $\Flex(c_0, n)$.
Moreover, the set $\Flex(c_0, n)$ is bounded.
\end{thm}
\begin{proof}
We will first show that $\Flex(c_0, n)$ is indeed closed under squares.
Let $c$ be a chamber in $\Flex(c_0, n)$ at Weyl distance $w$ from $c_0$ and
let $c_1$ and $c_2$ be chambers in $\Flex(c_0, n)$ adjacent to $c$, at Weyl distance $ws_i$ and $ws_j$ from $c_0$, respectively,
such that $|s_is_j|=2$ and $l(ws_i)=l(ws_j)=l(w)+1$.
Let $c_3$ be the unique chamber at Weyl distance $ws_is_j$ from $c_0$ that is $s_j$-adjacent to $c_1$ and $s_i$-adjacent to $c_2$.
\begin{figure}[ht!]
\centering
\scalebox{0.9}{
\begin{tikzpicture}[scale=0.8, line cap=round,line join=round,>=triangle 45,x=1.0cm,y=1.0cm]
\clip(-2.84,-2.5) rectangle (11,8.5);
\draw [shift={(3.02,-2.96)}] plot[domain=0.7512650838906699:2.3308834410494965,variable=\t]({1.*2.9008274681545605*cos(\t r)+0.*2.9008274681545605*sin(\t r)},{0.*2.9008274681545605*cos(\t r)+1.*2.9008274681545605*sin(\t r)});
\draw [shift={(3.02,-3.96)}] plot[domain=0.8674645121455721:2.168352779780061,variable=\t]({1.*6.5557303178211965*cos(\t r)+0.*6.5557303178211965*sin(\t r)},{0.*6.5557303178211965*cos(\t r)+1.*6.5557303178211965*sin(\t r)});
\draw [shift={(3.02,-3.96)}] plot[domain=0.910170346071863:2.10481737690763,variable=\t]({1.*8.409233020912193*cos(\t r)+0.*8.409233020912193*sin(\t r)},{0.*8.409233020912193*cos(\t r)+1.*8.409233020912193*sin(\t r)});
\draw [shift={(3.02,-3.96)}] plot[domain=0.9417135484806982:2.071676005884348,variable=\t]({1.*10.265066974939813*cos(\t r)+0.*10.265066974939813*sin(\t r)},{0.*10.265066974939813*cos(\t r)+1.*10.265066974939813*sin(\t r)});
\draw [line width=1.6pt] (3.040285010786088,-0.059243457589485)-- (3.052,1);
\draw [line width=1.6pt] (3.72,7.96)-- (0.044717335644615996,5.864423294381525);
\draw [line width=1.6pt] (3.72,7.96)-- (6.736718698965575,5.608573671804991);
\draw [line width=1.6pt] (3.72,7.96)-- (3.4230315098032142,6.2971519244917955);
\draw [line width=1.6pt] (0.044717335644615996,5.864423294381525)-- (1.2124275902570423,4.25266594861474);
\draw [line width=1.6pt] (6.736718698965575,5.608573671804991)-- (5.322260551648805,4.127941416227109);
\draw [line width=1.6pt] (5.322260551648805,4.127941416227109)-- (3.4230315098032142,6.2971519244917955);
\draw [line width=1.6pt] (3.4230315098032142,6.2971519244917955)-- (1.2124275902570423,4.25266594861474);
\draw [line width=1.6pt] (3.4230315098032142,6.2971519244917955)-- (3.260738512591466,4.4457863979853345);
\draw [line width=1.6pt] (3.260738512591466,4.4457863979853345)-- (3.22100147846369,2.592648197916276);
\draw [line width=1.6pt] (1.2124275902570423,4.25266594861474)-- (1.119356951721898,2.314165761520183);
\draw (3.02,8.64) node[anchor=north west] {$c_3=v_k$};
\draw (-0.62,6.6) node[anchor=north west] {$c_2$};
\draw (6.90,6.14) node[anchor=north west] {$c_1$};
\draw (3.52,6.85) node[anchor=north west] {$v_{k-1}$};
\draw (1.92,5.96) node[anchor=north west] {$s_i$};
\draw (1.2,7.6) node[anchor=north west] {$s_i$};
\draw (4.32,5.78) node[anchor=north west] {$s_j$};
\draw (5.1,7.48) node[anchor=north west] {$s_j$};
\draw (2.78,7.26) node[anchor=north west] {$s_k$};
\draw (-0.15,5.38) node[anchor=north west] {$s_k$};
\draw (6.1,5.18) node[anchor=north west] {$s_k$};
\draw (0.48,4.28) node[anchor=north west] {$d_2$};
\draw (5.30,4.1) node[anchor=north west] {$d_1$};
\draw (3.3,4.92) node[anchor=north west] {$v_{k-2}$};
\draw [line width=1.6pt] (3.260738512591466,4.4457863979853345)-- (1.119356951721898,2.314165761520183);
\draw [line width=1.6pt] (3.260738512591466,4.4457863979853345)-- (5.308664607439836,2.183257630496402);
\draw [line width=1.6pt] (5.308664607439836,2.183257630496402)-- (5.322260551648805,4.127941416227109);
\draw (1.70,4.06) node[anchor=north west] {$s_i$};
\draw (4.05,3.96) node[anchor=north west] {$s_j$};
\draw (2.82,2.58) node[anchor=north west] {$v_{k-3}$};
\draw (3.18,1.38) node[anchor=north west] {$v_{n+1}$};
\draw (3,-0.2) node[anchor=north west] {$v_n$};
\draw (8.3,5.54) node[anchor=north west] {$l(w)+1$};
\draw (7.72,3.64) node[anchor=north west] {$l(w)$};
\draw (7.06,2.06) node[anchor=north west] {$l(w)-1$};
\draw (5.1, 0) node[anchor=north west] {$n$};
\draw (3.28,-1.8) node[anchor=north west] {$c_0 = v_0$};
\begin{scriptsize}
\draw [fill=black] (3.02,-2.0) circle (2.5pt);
\draw [fill=black] (3.052,1) circle (2.5pt);
\draw [fill=black] (3.040285010786088,-0.08) circle (2.5pt);
\draw [fill=black] (3.72,7.96) circle (2.5pt);
\draw [fill=black] (0.044717335644615996,5.864423294381525) circle (2.5pt);
\draw [fill=black] (6.736718698965575,5.608573671804991) circle (2.5pt);
\draw [fill=black] (3.4230315098032142,6.2971519244917955) circle (2.5pt);
\draw [fill=black] (1.2124275902570423,4.25266594861474) circle (2.5pt);
\draw [fill=black] (5.322260551648805,4.127941416227109) circle (2.5pt);
\draw [fill=black] (3.260738512591466,4.4457863979853345) circle (2.5pt);
\draw [fill=black] (3.22100147846369,2.592648197916276) circle (2.5pt);
\draw [fill=black] (1.119356951721898,2.314165761520183) circle (2.5pt);
\draw [fill=black] (5.308664607439836,2.183257630496402) circle (2.5pt);
\end{scriptsize}
\end{tikzpicture}
}
\caption{Proof of Theorem~\ref{th:flex}}\label{fig:flex}
\end{figure}
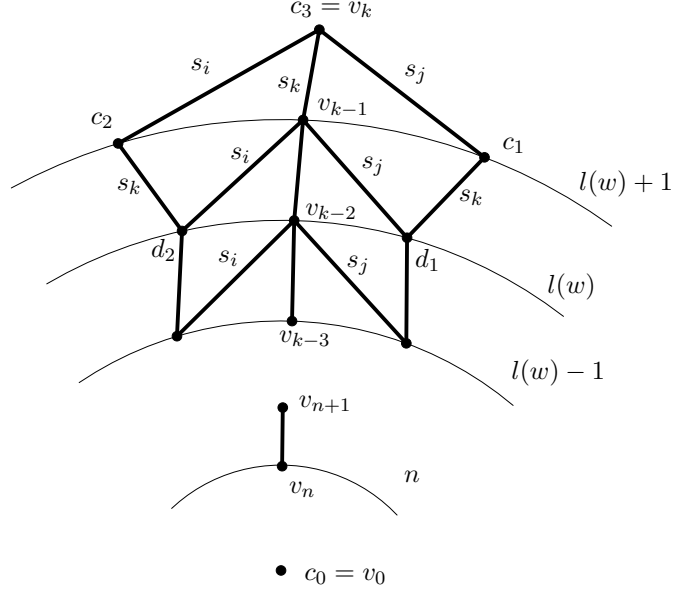

\noindent
Our aim is to show that also $c_3$ is an element of $\Flex(c_0, n)$.
If $l(ws_is_j)\leq n$, then this is obvious, so we may assume that $l(ws_is_j)>n$.

Let $\gamma = (c_0 = v_0, \dots, v_{n+1}, \dots, v_k = c_3)$ be an arbitrary minimal gallery between $c_0$ and $c_3$, as in Figure~\ref{fig:flex}
(so $k = l(w)+2 > n$).
We have to show that none of the chambers $v_{n+1},\dots,v_k$ is firm with respect to $c_0$.
This is clear for $v_k = c_3$.

If $k = n+1$, then there is nothing left to show, so assume $k \geq n+2$.
If $v_{k-1} \in\{ c_1, c_2 \}$, then $v_{k-1}$ is $n$-flexible by assumption, and since $k-1 > n$ it is not firm.
(In fact, this shows immediately that in this case, none of the chambers $v_{n+1},\dots,v_{k-1}$ is firm).
So assume that $v_{k-1}$ is distinct from $c_1$ and $c_2$;
then $v_{k-1}$ is $s_k$-adjacent to $c_3$ for some $s_k$ different from $s_i$ and~$s_j$.
Then by closing squares (Lemma~\ref{Squares}), we have $|s_js_k|=2$ and there is a chamber $d_1 \in \Sp(c_0, l(w))$
such that $d_1 \adj{s_j} v_{k-1}$ and $d_1 \adj{s_k} c_1$.
Similarly, there is a chamber $d_2 \in \Sp(c_0, l(w))$ such that $d_2 \adj{s_i} v_{k-1}$ and $d_2 \adj{s_k} c_2$.
Hence $v_{k-1}$ is not firm with respect to $c_0$.

Continuing this argument inductively (see Figure~\ref{fig:flex}), we conclude that none of the chambers $v_{n+1},\dots,v_k$ is firm with respect to $c_0$.
Hence $c_3$ is $n$-flexible; we conclude that $\Flex(c_0, n)$ is closed under squares with respect to $c_0$.

\smallskip

Conversely, let $\C$ be a set of chambers closed under squares that contains $\B(c_0, n)$;
we have to prove that $\Flex(c_0, n) \subseteq \C$.
So let $c \in \Flex(c_0, n)$ be arbitrary; we will show by induction on $k := \dw(c_0, c)$ that $c \in \C$.
This is obvious for $k \leq n$, so assume $k > n$.
Then $c$ is not firm, hence there exist $c_1, c_2\in \Sp(c_0, k-1)$ such that $c_1 \adj{s_1} c$ and $c_2 \adj{s_2} c$ for some $s_1 \neq s_2 \in S$.
By Lemma~\ref{Squares} we have $|s_1s_2|=2$ and there is $d\in \Sp(c_0, k-2)$ such that
$d \adj{s_2} c_1$ and $d \adj{s_1} c_2$.

Since $c$ is $n$-flexible and $c_1$, $c_2$ and $d$ all lie on some minimal gallery between $c_0$ and $c$,
it follows that also $c_1$, $c_2$ and $d$ are $n$-flexible.
By the induction hypothesis, all three elements are contained in $\C$.
Since $\C$ is assumed to be closed under squares, however, we immediately deduce that also $c \in \C$.

We conclude that $\Flex(c_0, n)$ is the square closure of $\B(c_0, n)$ with respect to $c_0$.

\smallskip

We finally show that $\Flex(c_0, n)$ is a bounded set.
Recall that a chamber $c$ is contained in $\Flex(c_0, n)$ if and only if $F^\#(\delta(c_0,c)) \leq n$.
By Theorem~\ref{th:boundforN}, there is a constant $d(n)$ such that $F^\#(\ow) > n$ for all $\ow \in W$ with $l(\ow) > d(n)$.
This shows that $\Flex(c_0, n) \subseteq \B(c_0, d(n))$ is indeed bounded.
\end{proof}

\begin{rem}
    In fact, the square closure of $\B(c_0, n)$ with respect to $c_0$ is precisely the convex hull of $\B(c_0, n)$.
    Indeed, the square closure is clearly contained in the convex hull.
    For the reverse inclusion, the crucial fact is that for any two chambers $c_1,c_2$ in $\B(c_0, n)$, there always exists
    a minimal gallery from $c_1$ to $c_2$ completely contained in $\B(c_0, n)$, and any two minimal galleries can be transformed
    into each other by a sequence of closing square operations; see \cite[Lemma~2.11]{DMSS} and its proof.
\end{rem}

\section{The automorphism group of a right-angled building}\label{se:AutRAB}

In this section, we study the group $\Aut(\Delta)$ of type-preserving automorphisms of a thick semi-regular right-angled building $\Delta$.
We will first study the action of a ball fixator and introduce root wing groups.
Then, when the building is locally finite, we will show that any proper open subgroup of $\Aut(\Delta)$ is a finite index subgroup of
the stabilizer of a proper residue; see Theorem~\ref{th:openpropermain}.

\begin{defn}\label{def:semi-regular}
Let $\Delta$ be a right-angled building of type $(W,S)$.
Then $\Delta$ is called \emph{semi-regular} if for each $s$, all $s$-panels of $\Delta$ have the same cardinality~$q_s$ of chambers.
In this case, the building is said to have \emph{prescribed thickness} $(q_s)_{s\in S}$ in its panels.

By \cite[Proposition 1.2]{HP2003}, there is a unique right-angled building of type $(W, S)$ of prescribed thickness $(q_s)_{s\in S}$
for any choice of cardinal numbers $q_s \geq 1$.
\end{defn}

\begin{thm}[{\cite[Theorem~B]{KT12}, \cite[Theorem 1.1]{Caprace}}]\label{th:caprace-main}
Let $\Delta$ be a thick semi-regular building of right-angled type $(W,S)$.
Assume that $(W,S)$ is irreducible and non-spherical.
Then the group $\Aut(\Delta)$ of type-preserving automorphisms of $\Delta$ is abstractly simple and acts strongly transitively on $\Delta$.
\end{thm}
Recall that a group $G$ acts \textit{strongly transitively} on a building $\Delta$ if it acts transitively on the pairs $(A, c)$ of an apartment $A$ and a chamber $c$ contained in $A$.
The strong transitivity has first been shown by Angela Kubena and Anne Thomas \cite{KT12} and has been reproved by Pierre-Emmanuel Caprace
in the same paper where he proved the simplicity \cite{Caprace}.
In our proof of Proposition~\ref{pr:trans-apt} below, we will adapt Caprace's proof of the strong transitivity to a more specific setting.

\medskip

The following extension result is very powerful and will be used in the proof of Theorem~\ref{th:fixedpointset} below.
\begin{prop}[{\cite[Proposition 4.2]{Caprace}}]\label{extending}
Let $\Delta$ be a semi-regular right-angled building. Let $s\in S$ and $\Pa$ be an $s$-panel.
Given any permutation $\theta \in \Sym(\ch(\Pa))$, there is some $g \in \Aut(\Delta)$ stabilizing $\Pa$ satisfying the following two conditions:
\begin{enumerate}[\rm (a)]
\item $g|_{\ch(\Pa)} = \theta$;
\item $g$ fixes all chambers of $\Delta$ whose projection on $\Pa$ is fixed by $\theta$.
\end{enumerate}
\end{prop}

\subsection{The action of the fixator of a ball in \texorpdfstring{$\Delta$}{Delta}}\label{ss:boundedsection}

In this section we study the action of the fixator $K$ in $\Aut(\Delta)$ of a ball $\B(c_0, n)$ of radius $n$ around a chamber $c_0$.
Our goal will be to prove that the fixed point set $\Delta^K$ coincides
with the square closure of the ball $\B(c_0,n )$ with respect to $c_0$, which is $\Flex(c_0, n)$, and which we know is bounded by Theorem~\ref{th:flex}.
%
%

\begin{thm}\label{th:fixedpointset}
Let $\Delta$ be a thick semi-regular right-angled building.
Let $c_0$ be a fixed chamber of $\Delta$ and let $n \in \N$.
Consider the pointwise stabilizer $K=\Fix_{\Aut(\Delta)}(\B(c_0, n) )$ in $\Aut(\Delta)$ of the ball $\B(c_0, n)$.

Then the fixed-point set $\Delta^K$ is equal to the bounded set $\Flex(c_0, n)$.
\end{thm}
\begin{proof}
Recall from Theorem~\ref{th:flex} that $\Flex(c_0, n)$ is precisely the square closure of $\B(c_0, n)$ with respect to $c_0$.
First, notice that the fixed point set of any automorphism fixing $c_0$ is square closed with respect to $c_0$
because the chamber ``closing the square'' is unique (see Definition~\ref{def:squareclosure}\eqref{it:sqcl}).
It immediately follows that $\Flex(c_0, n) \subseteq \Delta^K$.

We will now show that if $c$ is a chamber not in $\Flex(c_0, n)$, then there exists a $g \in K$ not fixing $c$.
Since $c$ is not $n$-flexible, there exists a chamber $d$ on some minimal gallery between $c_0$ and $c$ with $k := \dw(c_0, d) > n$ such that $d$ is firm.
Notice that any automorphism fixing $c_0$ and $c$ fixes every chamber on any minimal gallery between $c_0$ and $c$,
so it suffices to show that there exists a $g \in K$ not fixing $d$.

Since $d$ is firm, there is a unique chamber $e \in \Sp(c_0, k-1)$ such that $e \adj{s} d$ for some $s \in S$.
By Corollary~\ref{co:ballwing}, $\B(c_0, n) \subseteq X_s(e)$, where $X_s(e)$ is the $s$-wing of $\Delta$ corresponding to $e$.

Now take any permutation $\theta$ of $\Pa_{s,e}$ fixing $e$ and mapping $d$ to some third chamber $d''$ different from $d$ and $e$ (which exists because $\Delta$ is thick).
By Proposition~\ref{extending}, there is an element $g \in \Aut(\Delta)$ fixing $X_s(e)$ and mapping $d$ to $d''$.
In particular, $g$ belongs to~$K$ and does not fix $d$, as required.

We conclude that $\Delta^K = \Flex(c_0, n)$.
The fact that this set is bounded was shown in Theorem~\ref{th:flex}.
\end{proof}

\subsection{Root wing groups}\label{ss:rootwing}

In this section we define groups that resemble root groups, using the partition of the chambers of a right-angled building by wings;
we call these groups \emph{root wing groups}.

We show that a root wing group acts transitively on the set of apartments of $\Delta$ containing the given root.
We also prove that the root wing groups corresponding to roots disjoint from a ball $\B(c_0, n)$
are contained in the fixator of that ball in the automorphism group.

\smallskip

We first fix some notation for the rest of this section. Recall the notions from Definition~\ref{def:roots}.
\begin{notn}\phantomsection\label{not:A0}
    \begin{enumerate}[(i)]
    \item
        Fix a chamber $c_0\in \ch(\Delta)$ and an apartment $A_0$ containing $c_0$ (which can be considered as the fundamental chamber and the fundamental apartment).
        Let $\Phi$ denote the set of roots of $A_0$.
        For each $\alpha \in \Phi$, we write $-\alpha$ for the root opposite $\alpha$ in $A_0$.
    \item
        We will write $\app$ for the set of all apartments containing $c_0$.
        For any $A \in \app$, we will denote its set of roots by $\Phi_A$.
    \item
        For any $k \in \N$, we write $K_r := \Fix_{\Aut(\Delta)}(\B(c_0, r))$.
    \end{enumerate}
\end{notn}
\begin{defn}\phantomsection\label{def:rootwinggroup}
    \begin{enumerate}[(i)]
    \item
        When $\alpha \in \Phi_A$ is a root in an apartment $A$, its \emph{wall} $\partial\alpha$ consists of the panels of $\Delta$
        having chambers in both $\alpha$ and $-\alpha$.
        Since the building is right-angled, these panels all have the same type $s \in S$, which we refer to as the \emph{type of $\alpha$}
        and write as $\type(\alpha) = s$.
        Notice that the $s$-wings of $A$ are precisely the roots of $A$ of type $s$.
    \item
        Let $\alpha\in \Phi_A$ of type $s$ and let $c \in \alpha$ be such that $\Pa_{s,c} \in \partial\alpha$.
        Then we define the \emph{root wing group} $U_\alpha$ as
        \[ U_\alpha := U_s(c) :=  \Fix_{\Aut(\Delta)}(X_s(c)) . \]
        Observe that $U_\alpha$ does not depend of the choice of the chamber $c$ as all panels in the wall $\partial\alpha$ are parallel
        (see Definition~\ref{def:par}\eqref{it:par})
        and hence determine the same $s$-wings in $\Delta$.
    \end{enumerate}
\end{defn}

The fact that these groups behave, to some extent, like root groups in Moufang spherical buildings or Moufang twin buildings, is illustrated by the following fact.
\begin{prop}\label{pr:trans-apt}
    Let $\alpha \in \Phi_A$ be a root.
    The root wing group $U_\alpha$ acts transitively on the set of apartments of $\Delta$ containing $\alpha$.
\end{prop}
\begin{proof}
    We carefully adapt the proof of the strong transitivity of $\Aut(\Delta)$ from \cite[Proposition~6.1]{Caprace}.
    Let $c$ be a chamber of $\alpha$ on the boundary and let $A$ and $A'$ be two apartments of $\Delta$ containing $\alpha$.
    The strategy in \loccit (where $A$ and $A'$ are arbitrary apartments containing $c$)
    is to construct an infinite sequence of automorphisms $g_0, g_1, g_2, \dots$ such that
    \begin{compactenum}[(a)]
        \item each $g_n$ fixes the ball $\B(c, n-1)$ pointwise;
        \item let $A_n := g_n g_{n-1} \dotsm g_0(A)$; then $A_n \cap A' \supseteq \B(c,n) \cap A'$.
    \end{compactenum}
    We will show that the elements $g_i$ constructed in \loccit are all contained in~$U_\alpha$;
    the result then follows because $U_\alpha$ is a closed subgroup of $\Aut(\Delta)$.

    To construct the element $g_{n+1}$, we consider the set $E$ of chambers in $\B(c, n+1) \cap A'$ that are not contained in $A_n$ (as in \loccit).
    The crucial observation now is that by Theorem~\ref{th:fixedpointset}, the chambers of $E$ are firm with respect to $c$.
    Hence, for each $x \in E$, there is a unique chamber $y \in \Sp(c, n)$ that is $s$-adjacent to $x$ (for some $s \in S$).
    The element $g_{n+1}$ constructed in \loccit is then
    contained in the group generated by the subgroups $U_s(y)$ for such pairs $(y, s)$ corresponding to the various elements of $E$.
    However, because the elements of $E$ are firm, the root $\alpha$ is contained in each root corresponding to a pair $(y,s)$ in $A'$;
    \cite[Lemma 3.4(b)]{Caprace} now implies that each such group $U_s(y)$ is contained in $U_\alpha$.
\end{proof}
\begin{rem}
    The group $U_\alpha$ does \emph{not}, in general, act sharply transitively on the set of apartments containing $\alpha$.
    This is clear already in the case of trees: an automorphism fixing a half-tree and an apartment need not be trivial.
\end{rem}
\begin{coro}\label{co:refl}
    Let $\alpha \in \Phi_A$ be a root of type $s$ and let $c,c'$ be two $s$-adjacent chambers of $A$ with $c \in \alpha$ and $c' \in -\alpha$.
    Then there exists an element in $\langle U_\alpha, U_{-\alpha} \rangle$ stabilizing $A$ and interchanging $c$ and $c'$.
\end{coro}
\begin{proof}
    Let $A'$ be an apartment different from $A$ containing $\alpha$ (which exists because $\Delta$ is thick) and let $\beta$ be the root opposite $\alpha$ in $A'$.
    By Proposition~\ref{pr:trans-apt}, there is some $g \in U_\alpha$ mapping $-\alpha$ to $\beta$.
    Similarly, there is some $h \in U_{-\alpha}$ mapping $\beta$ to $\alpha$.
    Let $\gamma := h.\alpha$; then there exists a third automorphism $g' \in U_\alpha$ mapping $\gamma$ to $-\alpha$.
    The composition $g'hg \in U_\alpha U_{-\alpha} U_\alpha$ is the required automorphism.
\end{proof}

Next we present a property similar to the FPRS (``Fixed Points of Root Subgroups'') property introduced in \cite{superrigid} for groups with a twin root datum.
It is the analogous statement of \cite[Lemma 3.8]{CM13}, but in the case of right-angled buildings, we can be more explicit.

\begin{lem}\label{3.8}
For every root $\alpha\in \Phi$ with $\dist(c_0, \alpha) > r$, the group $U_{-\alpha}$ is contained in $K_r = \Fix_{\Aut(\Delta)}(\B(c_0, r))$.
\end{lem}
\begin{proof}
Let $\alpha$ be a root at distance $n > r$ from $c_0$ and let $s$ be the type of~$\alpha$.
Let $c$ be a chamber of $\alpha$ at distance $n$ from $c_0$ and let $c'$ be the other chamber in $\Pa_{s,c}\cap A_0$;
notice that $c' \in \Sp(c_0, n-1)$.
We will show that $\B(c_0, r)\subseteq X_s(c')$, which will then of course imply that $U_{-\alpha} = U_s(c') \subseteq K_r$.

The chamber $c$ is firm with respect to $c_0$ because if $c$ would be $t$-adjacent to some chamber at distance $n-1$ from $c_0$ for some $t \neq s$,
then $\partial\alpha$ would contain panels of type $s$ and of type $t$, which is impossible.
Corollary~\ref{co:ballwing} now implies that $\B(c_0, n-1) \subseteq X_s(c')$, so in particular $\B(c_0, r)\subseteq X_s(c')$.
\end{proof}

Following the idea of \cite[Lemmas 3.9 and 3.10]{CM13}, we present two variations on the previous lemma that allow us to transfer the results
to other apartments containing the chamber $c_0$.

\begin{lem}\label{3.9}
Let $g\in {\Aut(\Delta)}$ and let $A \in \app$ containing the chamber $d=gc_0$.
Let $b\in \Stab_{\Aut(\Delta)}(c_0)$ such that $A=bA_0$, and let $\alpha=b \alpha_0$ be a root of $A$ with $\alpha_0\in \Phi$.

If $\dist(d, -\alpha) > r$, then $bU_{\alpha_0}\inv{b}\subseteq gK_r\inv{g}$.
\end{lem}
\begin{proof}
Analogous to the proof of \cite[Lemma 3.9]{CM13}.
\end{proof}

\begin{defn}[{\cite[Section 2.4]{CM13}}]\label{essentialrootsdefn}
Let $w\in W$.
\begin{enumerate}[(i)]
\item A root $\alpha\in \Phi$ is called \emph{$w$-essential} if there is an $n\in \Z$ such that $w^n\alpha\subsetneq \alpha$.
\item A wall is called \emph{$w$-essential} if it is the wall $\partial\alpha$ of some $w$-essential root~$\alpha$.
\end{enumerate}
\end{defn}

\begin{lem}\label{3.10}
Let $A \in \app$ and let $b\in \Stab_{\Aut(\Delta)}(c_0)$ such that $A=bA_0$.
Also, let $\alpha=b\alpha_0$ (with $\alpha_0\in \Phi$) be a $w$-essential root for some $w\in \Stab_{\Aut(\Delta)}(A)/\Fix_{\Aut(\Delta)}(A)$.
Let $g\in \Stab_{\Aut(\Delta)}(A)$ be a representative of~$w$.

Then there exists some $n\in \Z$ such that
\begin{align*}
    U_{\alpha_0} &\subseteq \inv{b}g^{n} K_r g^{-n} b \quad\text{and} \\
    U_{-\alpha_0} &\subseteq \inv{b}g^{-n} K_r g^{n} b .
\end{align*}
\end{lem}
\begin{proof}
The proof can be copied ad verbum from \cite[Lemma 3.10]{CM13}.
\end{proof}

\subsection{Open subgroups of \texorpdfstring{$\Aut(\Delta)$}{Aut(Delta)}}\label{ss:open}

We now focus on the description of open subgroups of the automorphism group of $\Delta$.
The main result of this section will be that any proper open subgroup of the automorphism group
of a locally finite thick semi-regular right-angled building $\Delta$
is contained with finite index in the setwise stabilizer in $\Aut(\Delta)$ of a proper residue of $\Delta$
(see Theorem~\ref{th:openpropermain} below).

We will split the proof in the cases where the open subgroup is compact and non-compact.
The compact case is easy:



\begin{prop}\label{pr:openequivalent}
    Let $H$ be an open subgroup of $\Aut(\Delta)$.
    Then $H$ is compact if and only if it is a finite index subgroup of the stabilizer of a spherical residue of $\Delta$.
\end{prop}
\begin{proof}
This follows immediately from the fact that the maximal compact open subgroups of $\Aut(\Delta)$ are precisely the stabilizers of a maximal spherical residue of $\Delta$; see, for instance, \cite[Proposition 4.2]{DMSS}.
\end{proof}

\begin{quote}
    From now on, we assume that $H$ is a \textit{non-compact} open subgroup of~${\Aut(\Delta)}$.
\end{quote}

\begin{defn}\label{def:J}
We continue to use the conventions from Notation~\ref{not:A0} and
we will identify the apartment $A_0$ with $W$.
\begin{enumerate}[(i)]
    \item\label{it:na}
        Given a root $\alpha\in\Phi$, let $r_\alpha$ denote the unique reflection of $W$ setwise stabilizing the panels in $\partial\alpha$
        and let $U_\alpha$ be the root wing group introduced in Definition~\ref{def:rootwinggroup}.
        By Corollary~\ref{co:refl}, the reflection $r_\alpha \in W$ lifts to an automorphism
        $n_\alpha \in \langle U_\alpha, U_{-\alpha} \rangle \leq \Aut(\Delta)$ stabilizing $A_0$.
    \item
        For each $c \in \ch(\Delta)$ and each subset $J \subseteq S$, we write
        $\Rd_{J, c}$ for the residue of $\Delta$ of type $J$ containing $c$.
        We use the shorter notation $\Rd_J := \Rd_{J, c_0}$ when $c = c_0$.
        Moreover, we write $P_J := \Stab_{\Aut(\Delta)}(\Rd_J)$, and we call this a \emph{standard parabolic subgroup} of $\Aut(\Delta)$.
        Any conjugate of~$P_J$, i.e., any stabilizer of an arbitrary residue, is then called a \emph{parabolic subgroup}.
    \item\label{it:J}
        Let $J\subseteq S$ be minimal such that there is a $g\in {\Aut(\Delta)}$ such that $H \cap \inv{g} P_J g$ has finite index in $H$.
        In particular, $J$ is essential (see Definition~\ref{parb}\eqref{it:essential}). See also~\cite[Lemma 3.4]{CM13}.

        For such a $g$, we set $H_1 = g H \inv{g} \cap P_J$.
        Thus $H_1$ stabilizes $\Rd_J$ and it is an open subgroup of ${\Aut(\Delta)}$ contained in $gH\inv{g}$ with finite index;
        since $H$ is non-compact, so is $H_1$.
        Hence we may assume without loss of generality that $g = 1$ and hence $H_1 = H \cap P_J$ has finite index in $H$.
    \item
        Let $\app$ be the set of apartments of $\Delta$ containing $c_0$.
        For $A\in \app$ we let
        \[ N_A := \Stab_{H_1}(A) \quad \text{and} \quad W_A := N_A/\Fix_{H_1}(A) , \]
        which we identify with a subgroup of $W$.
        For $h\in N_A$, let $\overline{h}$ denote its image in $W_A \leq W$.
\end{enumerate}
\end{defn}

The idea will be to prove that $H_1$ contains a hyperbolic element $h$ such that the chamber $c_0$ achieves the minimal displacement of $h$.
Moreover, we can find the element $h$ in the stabilizer in $H_1$ of an apartment $A_1$ containing~$c_0$.
Thus we can identify it with an element $\overline{h}$ of $W$ and consider its parabolic closure (see Definition~\ref{parb}\eqref{it:parclo}).
The key point will be to prove that the type of $\Pc(\overline{h})$ is $J$, which will be achieved in Lemma~\ref{le:several}.

We will also show that $H_1$ acts transitively on the chambers of $\Rd_J$; this will allow us to conclude that any open subgroup of ${\Aut(\Delta)}$
containing $H_1$ as a finite index subgroup is contained in the stabilizer of $\Rd_{J\cup J'}$ for some spherical subset $J'$ of $J^\perp$ (Proposition~\ref{pr:main1}).

This strategy is analogous (and, of course, inspired by) \cite[Section~3]{CM13}.
As the arguments of \loccit are of a geometric nature, we will be able to adapt them to our setting.
The root groups associated with the Kac--Moody group in that paper can be replaced by the root wing groups defined in Section~\ref{ss:rootwing}.
It should not come as a surprise that many of our proofs will simply consist of appropriate references to arguments in~\cite{CM13}.

\begin{lem}\label{3.5}
For all $A \in \app$, there exists a hyperbolic automorphism $h\in N_A$ such that
\[
\Pc(\overline{h})=\langle r_\alpha \mid \alpha\text{ is an } \overline{h}\text{-essential root of } \Phi\rangle
\]
and is of finite index in $\Pc(W_A)$.
\end{lem}

\begin{proof}
    Using the fact that the reflections $r_\alpha$ lift to elements $n_\alpha \in \langle U_\alpha, U_{-\alpha} \rangle$
    (see Definition~\ref{def:J}\eqref{it:na}),
    the proof is the same as for~\cite[Lemma~3.5]{CM13}.
    Notice that by \cite[Lemma~2.7]{CM13}, the type of the parabolic subgroup $\Pc(\overline{h})$ is always essential
    (in the sense of Definition~\ref{parb}\eqref{it:essential}).
\end{proof}


\begin{lem}\label{3.7}
There exists an apartment $A \in \app$ such that the orbit $N_A.c_0$ is unbounded.
In particular, the parabolic closure in $W$ of $W_A$ is non-spherical.
\end{lem}
\begin{proof}
    The proofs of \cite[Lemmas 3.6 and 3.7]{CM13} continue to hold without a single change.
    Notice that this depends crucially on the fact that $H_1$ is non-compact.
\end{proof}


\begin{defn}\phantomsection\label{def:I}
\begin{enumerate}[\rm (i)]
    \item\label{it:A1}
        Let $A_1 \in \app$ be an apartment such that the essential component of $\Pc(W_{A_1})$ is non-empty and maximal with respect to this property
        (see Definition~\ref{parb}\eqref{it:essential});
        such an apartment exists by Lemma~\ref{3.7}.
        Choose $h_1\in N_{A_1}$ as in Lemma~\ref{3.5}.
        In particular, $h_1$ is a hyperbolic element of $H_1$.
    \item\label{it:I}
        Up to conjugating $H_1$ by an element of $\Stab_{\Aut(\Delta)}(\Rd_{J})$, we can assume without loss of generality that $\Pc(\overline{h_1})$
        is a standard parabolic subgroup that is non-spherical and has essential type $I$ $(\neq\emptyset)$.
        Moreover, the type $I$ is maximal in the following sense: if $A \in \app$ is such that $\Pc(W_A)$ contains a parabolic subgroup of essential type $I_A$
        with $I\subseteq I_A$, then $I=I_A$.
\end{enumerate}
\end{defn}

\begin{defn}
Recall that $\Phi$ is the set of roots of the apartment $A_0$.
For each $T\subseteq S$, let
\[ \Phi_T := \{ \alpha \in \Phi \mid \Rd_T \text{ contains at least one panel of } \partial\alpha \} \]
and
\[ L^+_T := \langle U_\alpha \mid \alpha \in \Phi_T\rangle , \]
where $U_\alpha$ is the root wing group introduced in Definition~\ref{def:rootwinggroup}.
\end{defn}

Our next goal is to prove that $H_1$ contains $L^+_J$, where $J$ is as in Definition~\ref{def:J}\eqref{it:J};
as we will see in Lemma~\ref{3.11} below, this fact is equivalent to $H_1$ being transitive on the chambers of $\Rd_{J}$.

We will need the results in Section~\ref{ss:rootwing} regarding fixators of balls and root wing groups.

\begin{notn}
    Since $H_1$ is open, we fix, for the rest of the section, some $r\in \N$ such that $\Fix_{\Aut(\Delta)}(\B(c_0, r)) \subseteq H_1$.
\end{notn}

The next lemma corresponds to \cite[Lemma 3.11]{CM13}, but some care is needed because of our different definition of the groups $U_\alpha$.
\begin{lem}\label{3.11}
Let $T\subseteq S$ be essential and let $A \in \app$.
Then the following are equivalent:
\begin{enumerate}[\rm (a)]
\item\label{it:ess-a}  $H_1$ contains $L^+_T$;
\item\label{it:ess-b}  $H_1$ is transitive on $\Rd_{T}$;
\item\label{it:ess-c}  $N_A$ is transitive on $\Rd_{T}\cap A$;
\item\label{it:ess-d}  $W_A$ contains the standard parabolic subgroup $W_T$ of $W$.
\end{enumerate}
\end{lem}
\begin{proof}
It is clear that \eqref{it:ess-c} and \eqref{it:ess-d} are equivalent.

We first show that \eqref{it:ess-a} implies \eqref{it:ess-c}.
It suffices to show that for each chamber $c_1$ of $A$ that is $s$-adjacent to $c_0$ for some $s \in T$,
there is an element of $N_A$ mapping $c_0$ to $c_1$.
Let $\alpha$ be the root of $A_0$ containing $c_0$ but not the chamber $c_2$ in $A_0$ that is $s$-adjacent to $c_0$;
notice that $U_\alpha$ and $U_{-\alpha}$ are contained in $L^+_T$.
By Proposition~\ref{pr:trans-apt}, there is some $g \in U_\alpha$ fixing $c_0$ and mapping $c_1$ to $c_2$.
Now the element $n_\alpha \in \langle U_\alpha, U_{-\alpha} \rangle$ stabilizes $A_0$ and interchanges $c_0$ and $c_2$;
it follows that the conjugate $g^{-1} n_\alpha g$ stabilizes $A$ and interchanges $c_0$ and $c_1$, as required.

The proofs of the implications \eqref{it:ess-d}~$\Rightarrow$~\eqref{it:ess-b}~$\Rightarrow$ \eqref{it:ess-a} are exactly as in \cite[Lemma~3.11]{CM13}.
\end{proof}

The next statement is the analogue of \cite[Lemma 3.12]{CM13}.

\begin{lem}\label{3.12}
Let $A \in \app$.
There exists $I_A\subseteq S$ such that $W_A$ contains a parabolic subgroup $P_{I_A}$ of $W$ of type $I_A$ as a finite index subgroup.
\end{lem}
\begin{proof}
The proof can be copied ad verbum from~\cite[Lemma 3.12]{CM13}.
%
%
\end{proof}

For each $A \in \app$, we fix such an $I_A\subseteq S$; without loss of generality, we may assume that $I_A$ is essential.
We also consider the corresponding parabolic subgroup $P_{I_A}$ contained in $W_A$.
Observe that $P_{I_{A_1}}$ has finite index in $\Pc(W_{A_1})$ by Lemma~\ref{lem2.4}, where $A_1$ is as in Definition~\ref{def:I}\eqref{it:A1}.
Therefore $I=I_{A_1}$.

The next task in the process of showing that $H_1$ contains $L^+_J$ is to prove that $J=I$,
which is achieved by the following sequence of steps, each of which follows from the previous ones and which are analogues of results in \cite{CM13}.
\begin{lem}\label{le:several}
Let $A \in \app$ and let $I$ and $J$ be as in Definition~\ref{def:I}\eqref{it:I} and~\ref{def:J}\eqref{it:J}, respectively. Then:
\begin{enumerate}[\rm (i)]
    \item $H_1$ contains $L^+_I$;
    \item $I_A\subset I$;
    \item $W_A$ contains $W_I$ as a subgroup of finite index;
    \item $I=J$.
\end{enumerate}
\end{lem}
\begin{proof}
\begin{enumerate}[\rm (i)]
\item
This follows from the fact that $I=I_{A_1}$ and $P_I=W_I$; the conclusion follows from Lemma~\ref{3.11}.

\item
See \cite[Lemma 3.14]{CM13}.

\item
See \cite[Lemma 3.15]{CM13}.
%

\item
See \cite[Lemma 3.16]{CM13}.
%
%
%
\qedhere
\end{enumerate}
\end{proof}

\begin{coro}\label{cor:trans-Rj}
The group $H_1$ acts transitively on the chambers of $\Rd_{J}$.
\end{coro}
\begin{proof}
This follows by combining Lemmas~\ref{3.11} and~\ref{le:several}.
\end{proof}

We are approaching our main result; the following proposition already shows, in particular, that $H$ is contained in the stabilizer of a residue,
and it will only require slightly more effort to show that it is a \emph{finite index} subgroup of such a stabilizer.
\begin{prop}\label{pr:main1}
Every subgroup of $\Aut(\Delta)$ containing $H_1$ as a subgroup of finite index is contained in a stabilizer $\Stab_{\Aut(\Delta)}(\Rd_{J\cup J'})$,
where $J'$ is a spherical subset of $J^\perp$.
\end{prop}
\begin{proof}
The proof is exactly the same as in \cite[Lemma 3.19]{CM13}.
\end{proof}
Notice that since $\Delta$ is irreducible, the index set $J\cup J'$ is only equal to $S$ if already $J=S$.


%
%

\begin{lem}\label{fi1}
The group $H_1$ is a finite index subgroup of $\Stab_{\Aut(\Delta)}(\Rd_J)$.
\end{lem}
\begin{proof}
Let $G := \Stab_{\Aut(\Delta)}(\Rd_J)$.
We already know that $H_1$ stabilizes $\Rd_J$ (see Definition~\ref{def:J}\eqref{it:J})
and acts transitively on the set of chambers of $\Rd_J$ (see Corollary~\ref{cor:trans-Rj}).
Notice that the stabilizer in $G$ of a chamber of $\Rd_J$ is compact, hence $H_1$ is a cocompact subgroup of $G$.
Since $H_1$ is also open in~$G$, we conclude that $H_1$ is a finite index subgroup of $G$.
\end{proof}

\begin{lem}\label{fi2}
For every spherical $J' \subseteq J^\perp$, the index of $\Stab_{\Aut(\Delta)}(\Rd_{J})$ in $\Stab_{\Aut(\Delta)}(\Rd_{J\cup J'})$ is finite.
\end{lem}
\begin{proof}
By \cite[Lemma 2.2]{Caprace}, we have $\ch(\Rd_{J \cup J'}) = \ch(\Rd_J) \times \ch(\Rd_{J'})$.
As $J'$ is spherical, the chamber set $\ch(\Rd_{J'})$ is finite; the result follows.
\end{proof}

We are now ready to prove our main theorem.

\begin{thm}\label{th:openpropermain}
Let $\Delta$ be a thick irreducible semi-regular locally finite right-angled building of rank at least $2$.
Then any proper open subgroup of $\Aut(\Delta)$ is contained with finite index in the stabilizer in $\Aut(\Delta)$ of a proper residue.
\end{thm}
\begin{proof}
Let $H$ be a proper open subgroup of $\Aut(\Delta)$.
If $H$ is compact, then the result follows from Proposition~\ref{pr:openequivalent}.

So assume that $H$ is not compact.
By Definition~\ref{def:J}\eqref{it:J}, we may assume that $H$ contains a finite index subgroup $H_1$ which, by Corollary~\ref{cor:trans-Rj},
acts transitively on the chambers of some residue $\Rd_J$.
By Proposition~\ref{pr:main1}, $H$ is a subgroup of $G := \Stab_{\Aut(\Delta)}(\Rd_{J\cup J'})$ for some spherical $J' \subseteq J^\perp$.
On the other hand, Lemmas~\ref{fi1} and~\ref{fi2} imply that $H_1$ is a finite index subgroup of $G$;
since $H_1$ is a finite index subgroup of $H$, it follows that also $H$ has finite index in $G$.

It only remains to show that $\Rd_{J \cup J'}$ is a proper residue.
If not, then $G = \Aut(\Delta)$, but since $G$ is simple (Theorem~\ref{th:caprace-main}) and infinite, it has no proper finite index subgroups.
Since $H$ is a proper open subgroup of $G$, the result follows.
\end{proof}

\section{Two applications of the main theorem}\label{se:appl}

In this last section we present two consequences of~Theorem~\ref{th:openpropermain},
both of which were suggested to us by Pierre-Emmanuel Caprace.
The first states that the automorphism group of a locally finite thick semi-regular right-angled building $\Delta$ is Noetherian (see Definition~\ref{Noet});
the second deals with reduced envelopes in $\Aut(\Delta)$.

\begin{defn}\label{Noet}
We call a topological group \emph{Noetherian} if it satisfies the ascending chain condition on open subgroups.
\end{defn}

We will prove that the group $\Aut(\Delta)$ is Noetherian by making use of the following characterization.
\begin{lem}[{\cite[Lemma 3.22]{CM13}}]\label{le:noeth}
Let $G$ be a locally compact group.
Then $G$ is Noetherian if and only if every open subgroup of $G$ is compactly generated.
\end{lem}

\begin{prop}\label{pr:noeth}
Let $\Delta$ be a locally finite thick semi-regular right-angled building. Then the group $\Aut(\Delta)$ is Noetherian.
\end{prop}
\begin{proof}
By Lemma~\ref{le:noeth}, we have to show that every open subgroup of $\Aut(\Delta)$ is compactly generated.
By Theorem~\ref{th:openpropermain}, every open subgroup of $\Aut(\Delta)$ is contained with finite index in the stabilizer of a residue of $\Delta$.

Stabilizers of residues are compactly generated, since they are generated by the stabilizer of a chamber $c_0$ (which is a compact open subgroup) together
with a choice of elements mapping $c_0$ to each of its (finitely many) neighbors.
Since a closed cocompact subgroup of a compactly generated group is itself compactly generated (see~\cite{MS59}), we conclude
that indeed every open subgroup of $\Aut(\Delta)$ is compactly generated and hence $\Aut(\Delta)$ is Noetherian.
\end{proof}

Our next application deals with reduced envelopes, a notion introduced by Colin Reid in \cite{Reid}
in the context of arbitrary totally disconnected locally compact (t.d.l.c.) groups.
\begin{defn}
\begin{enumerate}[(i)]
    \item
        Two subgroups $H_1$ and $H_2$ of a group $G$ are called \emph{commensurable} if $H_1 \cap H_2$ has finite index in both $H_1$ and $H_2$.
    \item
        Let $G$ be a totally disconnected locally compact (t.d.l.c.) group and let $H \leq G$ be a subgroup.
        An \emph{envelope} of $H$ in $G$ is an open subgroup of $G$ containing $H$.
        An envelope $E$ of $H$ is called \emph{reduced} if for any open subgroup $E_2$ with $[H : H\cap E_2] < \infty$ we have $[E : E\cap E_2] < \infty$.
\end{enumerate}
\end{defn}
Not every subgroup of $G$ has a reduced envelope, but clearly any two reduced envelopes of a given group are commensurable.

\begin{thm}[{\cite[Theorem B]{Reid2}}]\label{theomreidB}
Let $G$ be a t.d.l.c.\@ group and let $H$ be a (not necessarily closed) compactly generated subgroup of $G$.
Then there exists a reduced envelope for $H$ in $G$.
\end{thm}

We will apply Reid's result to show the following.
\begin{prop}\label{pr:env}
Every open subgroup of $\Aut(\Delta)$ is commensurable with the reduced envelope of a cyclic subgroup.
\end{prop}
\begin{proof}
Let $H$ be an open subgroup of $\Aut(\Delta)$ and assume without loss of generality that
$J \subseteq S$ and $H_1 = H \cap \Stab_{\Aut(\Delta)}(\Rd_{J})$ are as in Definition~\ref{def:J}\eqref{it:J}.
Let $h_1$ be the hyperbolic element of $H_1$ as in Definition~\ref{def:I}, so that $\Pc(\overline{h_1}) = W_J$.

By Theorem~\ref{theomreidB}, the group $\langle h_1 \rangle$ has a reduced envelope $E$ in $\Aut(\Delta)$.
In particular, $[E : E \cap H_1]$ is finite.

On the other hand, $H_2 := E \cap \Stab_{\Aut(\Delta)}(\Rd_{J})$ is an open subgroup of $G$ containing $\langle h_1 \rangle$,
hence Lemma~\ref{fi1} applied on $H_2$ shows that $H_2$ is a finite index subgroup of $\Stab_{\Aut(\Delta)}(\Rd_{J})$ for the same subset $J \subseteq S$, i.e.,
\[[ \Stab_{\Aut(\Delta)}(\Rd_{J}) : \Stab_{\Aut(\Delta)}(\Rd_{J}) \cap E]< \infty.\]
Since also $H_1$ has finite index in $\Stab_{\Aut(\Delta)}(\Rd_{J})$ by Lemma~\ref{fi1} again, it follows that also
$[H_1 : H_1 \cap E]$ is finite.
We conclude that $H_1$, and hence also $H$, is commensurable with $E$, which is the reduced envelope of a cyclic subgroup.
\end{proof}

\small
\bibliography{Primitive}

\def\cprime{$'$}
\begin{thebibliography}{BMPZ17}

\bibitem[AB08]{AbramenkoBrown}
Peter Abramenko and Kenneth~S. Brown.
\newblock {\em Buildings, Theory and applications}, volume 248 of {\em Graduate
  Texts in Mathematics}.
\newblock Springer, New York, 2008.

\bibitem[BM00]{BM-lattices}
Marc Burger and Shahar Mozes.
\newblock Lattices in product of trees.
\newblock {\em Inst. Hautes \'{E}tudes Sci. Publ. Math.}, (92):151--194 (2001),
  2000.

\bibitem[BMPZ17]{model-RAB}
Andreas Baudisch, Amador Martin-Pizarro, and Martin Ziegler.
\newblock A model-theoretic study of right-angled buildings.
\newblock {\em J. Eur. Math. Soc. (JEMS)}, 19(10):3091--3141, 2017.

\bibitem[Bou97]{Bourdon}
Marc Bourdon.
\newblock Immeubles hyperboliques, dimension conforme et rigidit\'e de
  {M}ostow.
\newblock {\em Geom. Funct. Anal.}, 7(2):245--268, 1997.

\bibitem[{Cap}14]{Caprace}
Pierre-Emmanuel {Caprace}.
\newblock Automorphism groups of right-angled buildings: simplicity and local
  splittings.
\newblock {\em Fund. Math.}, 224(1):17--51, 2014.

\bibitem[CM13]{CM13}
Pierre-Emmanuel {Caprace} and Timoth\'ee {Marquis}.
\newblock Open subgroups of locally compact {K}ac-{M}oody groups.
\newblock {\em Math. Z.}, 274(1-2):291--313, 2013.

\bibitem[CR09]{superrigid}
Pierre-Emmanuel {Caprace} and Bertrand {R\'emy}.
\newblock Simplicity and superrigidity of twin building lattices.
\newblock {\em Invent. Math.}, 176(1):169--221, 2009.

\bibitem[CT13]{CT13}
Inna Capdeboscq and Anne Thomas.
\newblock Cocompact lattices in complete {K}ac-{M}oody groups with {W}eyl group
  right-angled or a free product of spherical special subgroups.
\newblock {\em Math. Res. Lett.}, 20(2):339--358, 2013.

\bibitem[Dav98]{DavisCat0}
Michael~W. Davis.
\newblock Buildings are {$\mathrm{CAT}(0)$}.
\newblock In {\em Geometry and cohomology in group theory ({D}urham, 1994)},
  volume 252 of {\em London Math. Soc. Lecture Note Ser.}, pages 108--123.
  Cambridge Univ. Press, Cambridge, 1998.

\bibitem[DMSS18]{DMSS}
Tom De~Medts, Ana~C. Silva, and Koen Struyve.
\newblock Universal groups for right-angled buildings.
\newblock {\em Groups Geom. Dyn.}, 12(1):231--287, 2018.

\bibitem[HP03]{HP2003}
Fr{\'e}d{\'e}ric Haglund and Fr{\'e}d{\'e}ric Paulin.
\newblock Constructions arborescentes d'immeubles.
\newblock {\em Math. Ann.}, 325(1):137--164, 2003.

\bibitem[KT12]{KT12}
Angela Kubena and Anne Thomas.
\newblock Density of commensurators for uniform lattices of right-angled
  buildings.
\newblock {\em J. Group Theory}, 15(5):565--611, 2012.

\bibitem[MS59]{MS59}
Alexander~M. Macbeath and Stanis\l{}av \'Swierczkowski.
\newblock On the set of generators of a subgroup.
\newblock {\em Nederl. Akad. Wetensch. Proc. Ser. A 62 = Indag. Math.},
  21:280--281, 1959.

\bibitem[Rei16a]{Reid2}
Colin~D. Reid.
\newblock {Distal actions on coset spaces in totally disconnected, locally
  compact groups}.
\newblock {\em ArXiv e-prints}, October 2016.

\bibitem[Rei16b]{Reid}
Colin~D. Reid.
\newblock Dynamics of flat actions on totally disconnected, locally compact
  groups.
\newblock {\em New York J. Math.}, 22:115--190, 2016.

\bibitem[RR06]{RR06}
Bertrand R\'emy and Mark Ronan.
\newblock Topological groups of {K}ac-{M}oody type, right-angled twinnings and
  their lattices.
\newblock {\em Comment. Math. Helv.}, 81(1):191--219, 2006.

\bibitem[Tho06]{ThomasRAB}
Anne Thomas.
\newblock Lattices acting on right-angled buildings.
\newblock {\em Algebr. Geom. Topol.}, 6:1215--1238, 2006.

\bibitem[Tit69]{Titswords}
Jacques Tits.
\newblock Le probl\`eme des mots dans les groupes de {C}oxeter.
\newblock In {\em Symposia {M}athematica ({INDAM}, {R}ome, 1967/68), {V}ol. 1},
  pages 175--185. Academic Press, London, 1969.

\bibitem[TW11]{TW11}
Anne Thomas and Kevin Wortman.
\newblock Infinite generation of non-cocompact lattices on right-angled
  buildings.
\newblock {\em Algebr. Geom. Topol.}, 11(2):929--938, 2011.

\bibitem[Wei09]{Weiss}
Richard~M. Weiss.
\newblock {\em The structure of affine buildings}, volume 168 of {\em Annals of
  Mathematics Studies}.
\newblock Princeton University Press, Princeton, NJ, 2009.

\end{thebibliography}
\bibliographystyle{alpha}
\addcontentsline{toc}{section}{References}

\end{document}